\documentclass[a4paper,11pt]{amsart}

\usepackage[a4paper]{geometry}
\usepackage[OT2,T1]{fontenc}

\usepackage{lipsum} 

\usepackage[english]{babel}
\usepackage{amsmath,amssymb}
\usepackage{enumerate}
\usepackage{array}
\usepackage{mathrsfs}
\usepackage[utf8]{inputenc}
\usepackage[T1]{fontenc}
\usepackage{version}
\usepackage{lscape}
\usepackage[all]{xy}
\usepackage{datetime}
\usepackage{todonotes}
\usepackage{graphicx}
\usepackage{multirow}
\usepackage{hyperref}

\newtheorem{definition}{Definition}

\newtheorem{thm}{Theorem}[section]
\newtheorem{lm}[thm]{Lemma}
\newtheorem{propo}[thm]{Proposition}
\newtheorem{coro}[thm]{Corollary}
\newtheorem{conjecture}[thm]{Conjecture}
\newtheorem{rem}{Remark}

\newtheorem{nota}{Notation}

\def\Z{\mathbb{Z}}
\def\N{\mathbb{N}}
\def\Q{\mathbb{Q}}
\def\R{\mathbb{R}}
\def\E{\mathscr{E}}
\def\W{\mathscr{W}}
\def\P{\mathbb{P}}

\def\Res{\mathrm{Res}}

\newcommand{\cA}{\mathcal{A}}
\newcommand{\A}{{\rm A}}
\newcommand{\vv}{{\rm v}}
\newcommand{\ww}{{\rm w}}
\newcommand{\bZ}{\mathbb{Z}}

\newcommand{\bR}{\mathbb{R}}

\title{On the variation of the root number in families of elliptic curves}
\author{Julie Desjardins}

\begin{document}


\maketitle

\begin{abstract}
We prove the density of rational points on non-isotrivial elliptic surfaces by studying the variation of the root numbers among the fibers of these surfaces, conditionally to two analytic number theory conjectures (the squarefree conjecture and Chowla's conjecture). This is a weaker statement than one found in a preprint of Helfgott which proves (under the same assumptions) that the average root number is $0$ when the surface admits a place of multiplicative reduction. However, we use a different technique. The conjectures involved impose a restriction on the degree of the irreducible factors of the discriminant of the surfaces.

Moreover, we manage to drop the squarefree conjecture assumption under some technical hypotheses, and show thus unconditionally the variation of the root number on many elliptic surfaces, without imposing a bound for the degree of the irreducible factors. Under the parity conjecture, this guarantees the density of the rational points on these surfaces.
\end{abstract}


\section{Introduction}

Let $E$ be an elliptic curve defined over the field of rational numbers $\Q$.
The root number of $E$ is expressed as the product of the local factors: \[W(E)=\prod_{p\leq\infty}{W_p(E)},\]
where $p$ runs through the prime numbers and $\infty$ (representing the finite and infinite places of $\Q$), $W_p(E)\in\{\pm1\}$ and $W_p(E)=+1$ for all $p$ except a finite number of them. The \emph{local root number of $E$ at $p$} denoted by $W_p(E)$, is defined in terms of the epsilon factors of the Weil-Deligne representations of $\Q_p$ (see \cite{Del} and \cite{Tat}). Rohrlich \cite{Rohr} gives an explicit formula for the local root numbers in terms of the reduction of the elliptic curve $E$ at a prime $p\not=2,3$. Halberstadt \cite{Halb} gives tables (completed by Rizzo \cite{Rizz}) for the local root number at $p=2,3$ according to the coefficients of $E$. Observe moreover that we always have $W_\infty(E)=-1$.

The root number is equal to the sign $W(E)\in\{\pm1\}$ of the functional equation of $L(E,s)$ the $L$-function of $E$:
\begin{displaymath}
\mathscr{N}_E^{(2-s)/2}(2\pi)^{s-2}\Gamma(2-s)L(E,2-s)=W(E)\mathscr{N}_E^{s/2}(2\pi)^{-s}\Gamma(s)L(E,s).
\end{displaymath}

We restrict ourselves to elliptic curves defined over $\Q$. A large part of this work would extend without much difficulty to any number field $K$, but note that over general $K$ analytic continuation and functional equation for the $L$-function remain unknown (over $\Q$, this is guaranteed by Wiles' work \cite{Wiles} and its extention by Breuil, Conrad, Diamond, and Taylor \cite{bcdt-fermat}).

The Birch and Swinnerton-Dyer conjecture implies that the root number is related to the rank of the elliptic curve as follows:

\begin{conjecture}(Parity Conjecture)
\[W(E)=(-1)^{\mathrm{rank} \ E(\Q)}.\]
\end{conjecture}

As a consequence of this equality it suffices to have $W(E)=-1$ for the rank of $\ E(\Q)$ to be non-zero and in particular for $E(\Q)$ to be infinite.

Let $\E$ be an \textit{elliptic surface} over $\P^1$, i.e. a 2-dimensional projective variety endowed with a morphism $\pi: \E\rightarrow \P^1$ such that every fiber $\E_t=\pi^{-1}(t)$ is a non singular curve of genus 1 except a finite number of them. We also require that $\pi$ admits a section. That way, the elliptic surface can be seen as a family of elliptic curves. 

\begin{rem}
If we consider general elliptic surface
$\E$ over a smooth curve $C$, the case when the genus is $g(C)>1$ is uninteresting for our problem since $C(\Q)$ is finite. The case when $g(C)=1$ and $C(\Q)$ is infinite is interesting but for the moment we cannot handle it with our methods.
\end{rem}

Write the Weierstrass equation $y^2=x^3+A(T)x+B(T)$ where $A,B\in\Z[T]$. We suppose that it is a minimal Weierstrass model for $\E$, i.e. there are no irreducible polynomial $P$ such that $P^4\mid A$ and $P^6\mid B$. The discriminant $\Delta(T)=-16(4A(T)^3+27B(T)^2)$ corresponds to an homogeneous polynomial $\Delta_\E(U,V)=V^{12k}\Delta(U/V)$.
Here, $k$ is the smallest integer such that both $4k\geq\deg A$ and $6k\geq\deg B$ holds. Let $c_4(T),c_6(T)\in\left(\frac{1}{2\cdot3^3}\right)\Z[T]$ be the polynomial such that $A(T)=-27c_4(T)$ and $B(T)=-54c_6(T)$.

We write the factorisation into primitive\footnote{Each polynomial is irreducible and monic.} factors
$\Delta_\E(U,V)=c_\E\prod_{i=1}^r P_i(U,V)^{e_i}$, where $c_\E\in\Q$ and $e_i\in\N^*$. Each primitive homogeneous polynomial $P_i$ corresponds to a place of bad reduction of the surface $\E$ over $\Q(T)$. We denote by $\mathscr{B}=\mathscr{B}_\E$ the set of these polynomials of bad reduction. We will use Kodaira symbols to describe the type of their reduction more accurately (see \cite{Kodaira,Neronfibre}).
\begin{nota}
We will frequently use the two following polynomials defined from the factors of the discriminant:
\begin{enumerate} 
\item $B_\E(U,V)=\prod_{P\in\mathscr{ B}}{P(U,V)}$, the product of polynomials associated to places of bad reduction,
\item $M_\E(U,V)=\prod P(U,V)$  where $P$ runs through the polynomials associated to places of multiplicative reduction of $\E$.
\end{enumerate}
\end{nota}
We denote by $j_\E(T)=\frac{c_4(T)^3}{\Delta(T)}$ the rational function of the $j-invariant$ of the fibers of $\E$.
We distinguish the case where $\E$ is isotrivial, i.e. when $j_\E(T)$ is constant.

We consider the sets $W_+$ and $W_-$ given by \[W_\pm(\E)=\{t\in\Q: \E_t\text{ is an elliptic curve and }W(\E_t)=\pm1\}.\]

As a consequence of the parity conjecture, if $\#W_-(\E)=\infty$, then there exist infinitely many fibers of $\E$ that are non singular elliptic curves with positive rank, and this guarantees the density of the rational points on $\E$.

\subsection{Main results}

In this paper, we show the following theorem.

\begin{thm}\label{thmhelf}
Let $\E$ be a non isotrivial elliptic surface. Let $\Delta_\E(U,V)=c_\E\cdot P_1(U,V)^{e_1}\dots P_r(U,V)^{e_r}$ be the factorisation into primitive factors of the discriminant of $\E$. Suppose that 
\begin{enumerate}
\item\label{hyphelf4}
$\deg M_\E\leq3$, or $M_\E$ is the product of an arbitrary number of linear forms.
\item\label{hyphelf1} and every $P_i\in\mathscr{B}$ is such that $\deg P_i\leq6$, except those of type $I_0^*$;
\end{enumerate}

Then the sets $W_\pm$ are both infinite.

Moreover, if one assumes the parity conjecture, then the rational points of $\E$ are Zariski-dense.

\end{thm}

This is a weaker statement than the one found in Helfgott's preprint \cite{Helfgott} which proves (under the same assumptions) that the average root number is $0$ when the surface admits a place of multiplicative reduction. However, the proof of our theorem is somehow more direct and less scattered.

The two assumptions in Theorem \ref{thmhelf} correspond to known cases of conjectures in analytic number theory, namely :
\begin{enumerate}
\item Hypothesis \ref{hyphelf4} is used to ensure that the polynomial $M_\E$ satisfies Chowla's conjecture (stated further on),
\item Hypothesis \ref{hyphelf1} is used to ensure that the polynomial $B_\E$ satisfies the squarefree conjecture (also stated further on).
\end{enumerate}
So Theorem \ref{thmhelf} can also be seen as saying that the Squarefree and Chowla's conjectures imply variation of the root number for a non isotrivial family.

Until now, nothing was known about the variation of the root number on surfaces with places of bad reduction (and not $I_0^*$) whose polynomials have arbitrarily large degree. To partly solve this question, we prove the following theorem:

\begin{thm}\label{nouveauthm}
Let $\E$ be a non-isotrivial elliptic surface. Let again $\Delta_\E(U,V)=c_\E\cdot P_1(U,V)^{e_1}\dots P_r(U,V)^{e_r}$ be the factorisation into primitive factors of the discriminant of $\E$. Suppose that 
\begin{enumerate}
\item\label{hyp2} for all $P_i$ with reduction of type $II$, $II^*$, $IV$ or $IV^*$, one has
\[\mu_3\subseteq\Q[T]/P_i(T,1),\]
where $\mu_3$ is the group of third roots of unity.
\item\label{hyp3} for all $P_i$ with reduction of type $III$ or $III^*$ one has 
\[\mu_4\subseteq\Q[T]/P_i(T,1),\]
where $\mu_4$ is the group of fourth roots of unity.
\item\label{hyp4}
$\deg M_\E\leq3$, or $M_\E$ is the product of an arbitrary number of linear forms,
\item\label{hyp1} and every $P_i$ of reduction of type $I_m^*$ ($m\geq1$) is such that $\deg P_i\leq6$.
\end{enumerate}

Then the sets $W_\pm(\E)$ are both infinite.

Moreover, if one assumes the parity conjecture, then the rational points of $\E$ are Zariski-dense.
\end{thm}


There are families of elliptic surfaces whose coefficients have factors of unbounded degree verifying the hypotheses of Theorem \ref{nouveauthm}. The following Corollary \ref{thmexemple} provides such examples.

\begin{coro}\label{thmexemple}
Let $Q\in\Q[T]$ be a squarefree polynomial such that its irreducible factors have degree less or equal to 6 and not equal to $T$. Let $N\in\N^*$. Put \[P(T)=3\alpha^2Q(T)^2+\beta^2T^{2N},\]
and $\alpha,\beta\in\Z$ coprime.

Let $\E$ be the elliptic surface given by the equation \[\E:y^2=x^3-27P(T)Q(T)^2x-54\beta P(T)Q(T)^3T^N.\]

Then $W_+$ and $W_-$ are infinite.

Moreover, if we assume the parity conjecture to hold, then the rational points of $\E$ are Zariski-dense.
\end{coro}

In this example, one has $\deg P=2\max (\deg Q,N)$, which can be as large as we want.

\subsection{Previous results}

Rohrlich pioneered the study of variations of root numbers on algebraic families of elliptic curves in \cite{Rohr}. Many authors followed: see, for instance, \cite{Manduchi}, \cite{GM}, \cite{Rizz}, \cite{HCC}, \cite{Helfgott}, \cite{VA}.

Some authors (notably \cite{CS},\cite{VA}) remarked that it can happen that the root number of the fibers takes always the same value when the elliptic surface is isotrivial, i.e. its modular invariant $j_\E$ has no $T$-dependence. For this reason, we restrict our attention in this paper to non-isotrivial elliptic surfaces (leaving the isotrivial case for another paper \cite{Desjardins4}).

We use Rohrlich's formula for local root numbers and a study of the monodromy of the singular fibers with Tate's algorithm \cite{Tatefibre} to assemble a root number formula for a general elliptic surface (see Theorem \ref{formuledusigne}). This formula splits into different parts corresponding to the "contribution" of a place of bad reduction on the surface, in a way previously studied by Manduchi \cite{Manduchi} and Helfgott \cite{Helfgott}. 

In Helfgott's unpublished paper \cite{Helfgott}, the Squarefree conjecture and Chowla's conjecture are used to prove that the average root number over $\Q$ is $av_\Q W(\E_t)=0$, on an elliptic surface $\E$ with at least one place of multiplicative reduction (i.e. with $M_\E\not=1$). When the surface admits no place of multiplicative reduction (i.e. when $M_\E=1$) then it was stated in \cite{HCC} that $-1<av_\Q W(\E_t)<+1$. The author completes the demonstration of this result and reviews Helfgott's paper in her phD thesis \cite{Desjardinsthese}.

We then combine our formula for root number with an adaption of a sieve introduced by Gouvêa, Mazur, and Greaves \cite{GM}, \cite{Greaves}, already improved by V\'arilly-Alvarado \cite{VA}. Our modified sieve Corollary \ref{sflcrible}, deduced from Theorem \ref{conjecturesfacteurs}, allows us to search for infinitely many pairs of fibers on a surface that have opposite root numbers. This proves our density results (Theorems \ref{thmhelf} and \ref{nouveauthm}). For a similarly motivated idea, see \cite{Manduchi} and \cite{VA}.
 
\subsection{Outline of the paper}

In section \ref{sectionconjectures}, we present the two analytic number theory conjectures
 and prove  Theorem \ref{conjecturesfacteurs}, an auxiliary result which has its own interest: a way to make the conjectures work at the same time. In section \ref{sectionformuledusigne}, we present a formula for the root number of the fibers of an elliptic surface. In section \ref{sectionvariation}, we use our formula to give conditions on the pairs of coprime integers $[m_1,n_1]$ and $[m_2,n_2]$ under which $\E_{\frac{m_1}{n_1}}$ and $\E_{\frac{m_2}{n_2}}$ have opposite root number (Lemmas \ref{VariationIm} and \ref{variationglobale}). In section \ref{sectionproofHelfgott}, we use the sieve of Corollary \ref{sflcrible} obtained from Theorem \ref{conjecturesfacteurs} to prove Theorems \ref{thmhelf} and \ref{nouveauthm}. In section \ref{sectionexample}, we prove Corollary \ref{thmexemple}, which gives examples of surfaces satisfying the hypotheses of Theorem \ref{nouveauthm} which have coefficients with irreducible factors of arbitrarily large degree. 

\subsection{Acknowledgments} The author is very grateful to her supervisor M. Hindry for suggesting this beautiful problem, and for his support in proving Theorem \ref{conjecturesfacteurs}. She would also like to thank R. de la Bretèche for carefully checking an earlier version of the proof of this theorem. The author also benefitted from conversations with K. Destagnol and R. Griffon, and from correspondence with H. Helfgott. The manuscript was greatly improved by insightful comments from D. Rohrlich. Finally, she thanks the anonymous referee for several useful suggestions.

\section{Two analytic number theory conjectures}\label{sectionconjectures}

In this section, we treat simultaneously the cases of a polynomial $f$ of degree $d$ with integer coefficients, either in one variable, or homogeneous in two variables - this way either we have $f(T)=a_0+\dots+a_dT^d\in\bZ[T]$ or $f(U,V)=a_0V^d+\dots+a_dU^d\in\bZ[U,V]$. We denote by $h=1$ or $2$ the number of variables and $\vv$ a vector with integer coefficients in $\bZ^h$, i.e. $\vv\in\bZ$ or $\vv\in\bZ^2$.
We study two properties describing the factorisation of $f(\vv)$ : the first describes the proportion of squarefree values, the second the parity of the number of prime factors.

It is natural to assume that $f$ is {\it primitive} (i.e. that its content is equal to 1) and that $f$ is squarefree (i.e. that there is no polynomial $f_0$ such that $f_0^2\mid f$), which is the same as supposing that its discriminant $D_f$ is not zero. We will make these assumptions throughout the present section.

We denote by $\cA$ an arithmetic progression\footnote{The expression  "arithmetic progression" is usually only used in the case where $h=1$. In the case $h=2$, $\cA$ is a "translated lattice".} of the form
\[\cA:=\{a+Nt\;|\; t\in\bZ\}\quad{\rm or}\quad \cA:=\{(a+Nu,b+Nv)\;|\; (u,v)\in\bZ^2\},\]
where $N\not=0$ in the first case and $a$ and $b$ are coprime to $N\not=0$ in the second case. One can also write $\cA=\phi(\bZ)$ with $\phi(t)=a+Nt$ or 
$\cA=\phi(\bZ^2)$ with $\phi(u,v)=(a+Nu,b+Nv)$.

We denote by $|\cdot|$ the usual absolute value on $\bR$ or the max norm on $\bR^2$. We also introduce the notations
\[\bZ^h(X)=\left\{\vv\in\bZ^h\;|\;|\vv|\leq X\right\},\;\; \cA(X):=\left\{\vv\in\cA\;|\;|\vv|\leq X\right\}\quad{\rm and}\quad \A(X):=\sharp\cA(X),\]
so that $\A(X)$ is roughly proportional to $X^h$. 
More precisely, we have $\A(X)\sim \left(\frac{X}{N}\right)^h$.

\subsection{Squarefree conjecture}

We want to estimate the proportion of squarefree values in an arithmetic progression. 

\noindent{\bf Notations.} Put 
 $\delta_{f}:=\gcd\left\{f(\vv)\;|\;\vv\in\bZ^h\right\}$, then  let $d_{f}$ be the smallest integer such that $\delta_f/d_{f}$ is squarefree. Write
 $d_{f}=\prod_pp^{\nu_p}$.
We denote by $t_f({p})$ the number of solutions modulo $p^{2+\nu_p}$ of $f(\vv)d_{f}^{-1}\equiv 0\mod p^2$, or $f(\vv)\equiv 0\mod p^{2+\nu_p}$ and we put
\[C_f:=\prod_p\left(1-\frac{t_f({p})}{p^{h(2+\nu_p)}}\right)\]

Note that 
the product defining $C_f$ is absolutely convergent since for $h=1$ (resp. $h=2$) one has $t_f(p)=O(1)$ (resp. $t_f(p)=O(p^2)$).

We define $Sqf(X)$ to be the number of elements $\vv\in\Z^h(X)$ such that $f(\vv)d_{f,\cA}^{-1}$ is squarefree.

\begin{conjecture}\label{conjsfc3} (Squarefree Conjecture) 
Let $f$ be a primitive squarefree polynomial with integer coefficients. Then we have
\begin{equation}\label{nbdesqf2} Sqf(X)
=C_{f}X^h+o(X^h)
\end{equation}
\end{conjecture}

Observe that the following statement is equivalent to Conjecture \ref{conjsfc3}\footnote{A proof of this equivalence can be found in \cite[1.3.1.]{Desjardinsthese}.}.

\begin{conjecture}\label{sfc2} (Squarefree Conjecture, alternative version) 
\begin{equation}
\sharp\left\{\vv\in\bZ^h(X)\;|\; \text{there exists $p>X^{h/2}$, such that $p^2$ divides $f(\vv)$}\right\}=o(X^h).
\end{equation}
\end{conjecture}

To study the almost\footnote{We allow the squarefree part of the values of $f$ to be a certain fixed integer.} squarefree values we proceed as follows. Let $\cA=(a,b)+
N\Z^2\subset\Z^2$ (or $\cA=a+N\Z\subset\Z$) be an arithmetic progression. We put
$g:=f\circ\phi$, $d_{f,\cA}:=d_{g}$, 
\[C_{f,\cA}:=\frac{1}{N^h}\prod_{p\nmid N}{\left(1-\frac{t_f(p)}{p^{h(2+\nu_p)}}\right)}\] 
and let $Sqf'(X)$ to be the number of elements $\vv\in\cA(X)$ such that $f(\vv)$ is squarefree.
We can write the Squarefree conjecture in a third form:

\begin{conjecture}\label{sfc1} (Squarefree Conjecture on arithmetic progressions)
\begin{equation}\label{nbdesqfA}Sqf'(X)=C_{f,\cA}X^2+o(X^2)
\end{equation}
\end{conjecture}

Note that Conjectures \ref{conjsfc3}, \ref{sfc2} and \ref{sfc1} are equivalent to one another.

\begin{thm}\label{theosqf}  (Hooley \cite{Hool}, Greaves \cite{Greaves}
) Let $f$ be a primitive squarefree polynomial with integer coefficients in $h$ variables ($h=1$ or $2$). Suppose that every irreducible factors of $f$ has degree less or equal to $3h$. 
Then $f$ verifies the Squarefree conjecture.
\end{thm}

\medskip

Note also that we have unconditionally the following estimate:

\begin{propo}\label{1.13}
Let $p$ be a prime number and let $f$ be a primitive squarefree polynomial with integer coefficients in $h$ variables ($h=1$ or $2$). We have

\begin{equation}\label{borne2}\sharp\left\{\vv\in \bZ^h(X)\;|\; p^2\text{ divides } f(\vv)\right\}\ll \frac{X^h}{p^2}+X^{h-1}.
\end{equation}
Moreover, the majoration implicit in this formula depends only of $f$, and not of $p$.
\end{propo}

\begin{proof}
Let $I$ be an integer interval of length $p^2$, i.e. $I=[M+1,M+p^2]$, or a product of at most two of such intervals, i.e. $I=[M+1,M+p^2]\times[M'+1,M'+p^2]$. We naturally have
\[\sharp\left\{\vv\in I\;|\; p^2\text{ divides } f(\vv)\right\}=t_f({p}).\]
This is $O(1)$ when $h=1$ and $O(p^2)$ when $h=2$. Taking a sum of $N$ such interval allows to obtain:
\[\sharp\left\{\vv\in \bZ^h(Np^2)\;|\; p^2\text{ divides } f(\vv)\right\}=t_f({p})(2N)^h.\]
Since $X\leq p^2(\lfloor\frac{X}{p^2}\rfloor+1)$, we have 
\begin{equation}\label{borne1}\sharp\left\{\vv\in \bZ^h(X)\;|\; p^2\text{ divides } f(\vv)\right\}\leq t_f({p})\left(2\frac{X}{p^2}+O(1)\right)^h
\end{equation}
For $h=1$, this estimation suffices. For $h=2$, first remark that the estimation gives $p^2(\frac{X^2}{p^4}+\frac{X}{p^2}+1)$, so the proposition is shown for primes such that $p^2\leq X$. For the big $p$, one can introduce the following elementary refinement  (see \cite[Lemma 1]{Greaves}). 

Put $Z_p=\{\omega\mod p^2 \mid f(\omega,1)\equiv 0\mod p^2\}$.
Then the set of solutions of $f(a,b)\equiv 0\mod p^2$ can be divided up into the set $L_{0}=p\bZ^2$ and the arithmetic progressions $L_{\omega}=\{\vv=(a,b)\in\bZ^2\;|\; a\equiv \omega b\mod p^2\}$ with $\omega\in Z_p$. We trivially have \[\sharp\{\vv\in L_0 \ \vert \ \vert\vv\vert\leq X\}=O\left(\frac{X^2}{p^2}\right).\]
Since the lattices $L_{\omega}$ have index $p^2$ in $\bZ$, we have \begin{equation}\label{poulet}\sharp\{\vv\in L_\omega\ \vert \ \vert\vv\vert\leq X \}=O\left(\frac{X^2}{p^2}+X\right).\end{equation}
Indeed, since we have $\vv=(\omega u+p^2v,v)$ with $(u,v)\in\Z^2$, the conditions can be written as $\vert v\vert\leq X$ and $\vert \omega u+p^2v\vert\leq X$. For a given $v\in[-X,X]$ we have thus $u\in[\frac{-\omega v}{p^2}-\frac{X}{p^2},\frac{-\omega v}{p^2}+\frac{X}{p^2}]$ thus $\frac{2X}{p^2}+O(1)$ solutions. We obtain equation \ref{poulet} by summing over $v\in[-X,X]$.
\end{proof}


The Squarefree conjecture leads to a sieve allowing to find infinitely many values of a polynomial whose square factors part is a given constant. We use the following version of the sieve, introduced by V\'arilly-Alvarado in \cite{VA}.

\begin{coro}{\cite[Corollary 5.8]{VA}}\label{crible}
Let $f(U,V)\in\Z[U,V]$ be an homogeneous polynomial in two variables of degree $d$. Assume that no square of a nonunit in $\Z[U,V]$ divides $f(U,V)$, and that no irreducible factor of $f$ has degree greater than $6$. Fix
\begin{itemize}
\item a sequence $S=(p_1,\dots,p_s)$ of distinct prime numbers and
\item a sequence $T=(t_1,\dots,t_s)$ of nonnegative integers.
\end{itemize}
Let $N$ be an integer such that $p^2\mid N$ for all primes $p<\deg f$ and $p_1^{t_1+1}\dots p_s^{t_s+1}\mid N$. Suppose that there exist integers $a$, $b$ such that 

\[f(a,b)\not\equiv0\mod p^2\text{, whenever }p\mid N \text{ and }p\not=p_i\text{ for any }i,\]
 and such that
\[v_{p_i}(f(a,b))=t_i\text{ for every }i=1,\dots,s.\]
Then there are infinitely many pairs of integers $(u,v)$ such that
\[u\equiv a\mod N, v\equiv b\mod N,\]
and
\[f(u,v)=p_1^{t_1}\dots p_s^{t_s}\cdot l,\]
where $l$ is squarefree and $v_{p_i}(l)=0$ for all $i=1,\cdots,s$.
\end{coro}

We will refine this sieve in Corollary \ref{sflcrible} by imposing further conditions on the values $f(u,v)$ and an auxiliary polynomial $g(u,v)$.

\subsection{Chowla's conjecture}

The second conjecture studies the proportion of the values $f(\vv)$ with a certain parity of the number of prime factors. Recall the definition of Liouville's function.

\begin{definition} For a non-zero integer $n=\prod_pp^{\nu_p(n)}$, we denote by $\Omega(n)=\sum_p\nu_p(n)$ the number of its prime factors and we define {\em Liouville's function} by the formula
\[\lambda(n)=(-1)^{\Omega(n)}.\]
\end{definition}

\begin{rem} Remember that Moebius' function is defined as \[\mu(n)=\begin{cases}
1& \text{if $n$ has an even number of distinct prime factors}\\
0 & \text{if $n$ has a square factor}\\
-1& \text{if $n$ has an odd number of distinct prime factors}
\end{cases}\]

Liouville's function resembles Moebius' function, but differs in the presence of a square factor. More precisely, the relation is the following: $\mu(n)=\lambda(n)$ if $n$ is squarefree and $\mu(n)=0$ if there exists $p^2$ dividing $n$.
\end{rem}

\begin{conjecture} (Chowla's conjecture) Let $f$ be a primitive squarefree polynomial with integer coefficients. The following estimation holds for every arithmetic progression $\cA$:
\begin{equation}
\sum_{\vv\in\cA(X)}\lambda(f(\vv))=o(X^2)\end{equation}
\end{conjecture}

\medskip

The known results up to now are the following:

\begin{thm}\label{theochowla}  Let $f$ be a primitive squarefree polynomial of degree $d$ with integer coefficients, in $h$ variables ($h=1$ or $2$). Chowla's conjecture holds in the following cases:
\begin{enumerate}
\item (Hadamard -- de la Vall\'ee Poussin) $h=1$ and $d=1$.
\item (Helfgott \cite{HelfChowla}, Lachand \cite{Lachandthese}\nocite{Lachand1}) $h=2$ and $d\leq 3$.
\item (Green-Tao \cite{GT}) $h=2$ and $f$ is a product of linear forms.
\end{enumerate}
\end{thm}

\subsection{Combination of the conjectures}

In the course of the proof of Theorem \ref{nouveauthm}, we will use the following analytical result which does not seem to appear in the literature.



\begin{thm}\label{conjecturesfacteurs} Fix an element $\epsilon\in\{1,-1\}$. Let $f,g\in\Z[U,V]$ be two squarefree primitive homogeneous polynomials with no common factor. Assume the Squarefree conjecture hold for $f$ and $g$ and that Chowla's conjecture hold for $f$. Then for any arithmetic progression $\cA=N\Z^2+(a,b)$ (where $a$, $b$, $N\not=0$  are integers such that $a$ and $b$ are coprime to $N$), define $T(X)$ to be the number of pairs of integers $(u,v)\in\cA(X)$ such that
\begin{enumerate}
\item $\frac{f(u,v)g(u,v)}{d_{fg}}$ is not divisible by $p^2$ for any prime $p$ such that $p\nmid N$;
\item $\lambda(f(u,v))=\epsilon$
\end{enumerate}

Then
\begin{equation}
T(X)=\frac{C_{fg,\cA}}{2}X^2+o\left(
X^2\right).
\end{equation}
\end{thm}

\begin{rem}
This indicates a sort of independence between the two properties of the values of a polynomial
\begin{itemize}
\item being "squarefree"
\item "parity" of the number of factors.
\end{itemize}
\end{rem}

\begin{rem}\label{thmsurlessecteurs}
Let $\{B_i\}_{0\leq i\leq r}$ be lines in $\R^2$ passing through $(0,0)$, and denote by $S_i$, $0\leq i \leq n+1$, the connected components of $\R^2-\cup_{0\leq n+1}{B_i}$.

If, instead of assuming Chowla's conjecture, one assumes the stronger fact that for $S_i$ and for any $\cA=N\Z^2+(a,b)$ as previously one has
\begin{equation}\label{Chowlasursecteurs}\sum_{v\in\cA(X)\cap S_i}{\lambda(f(u,v))}=o(X^2),\end{equation}
then the same proof as in Theorem \ref{conjecturesfacteurs} allows one to obtain a similar but stronger estimation for the counting function \[T_{S_i}(X)=\{(u,v)\in \cA(X)\cap S_i \ \vert \ \text{(u,v) respects (i) and (ii)}\}\]
precisely:
\begin{equation}
T_{S_i}(X)=\frac{C_{fg,\cA}}{2}X^2+o\left(
X^2\right).
\end{equation}

Indeed, the Squarefree conjecture on arithmetic progressions (Conjecture \ref{sfc1}) on a polynomial $f$ implies the Squarefree conjecture on $\cA(X)\cap S_i$ for any $0\leq i\leq n+1$. Moreover, Chowla's Conjecture on $\cA(X)\cap S_i$ (given by Equation (\ref{Chowlasursecteurs})) for every $0\leq i\leq n+1$ is proven when $deg f\leq 3$ (see \cite{HelfChowla}, \cite{HelfChowla2}) and when $f$ is a finite product of linear polynomial (see \cite{GT}).
\end{rem}

Combining Theorem \ref{conjecturesfacteurs} with Theorem \ref{theosqf} and Theorem \ref{theochowla}, we obtain:
\begin{coro}
Fix $\epsilon\in\{1,-1\}$. Let $f,g$ be squarefree homogeneous polynomial in two variables with integer coefficients. Assume that every factor of $g$ has degree less of equal to 6 and, either that $\deg f\leq 3$, or that $f$ is a product of linear forms. For an arithmetic progression $\cA$, let $T(X)$ be the counting function as defined in Theorem \ref{conjecturesfacteurs}. Then the following estimate holds
\begin{equation}
T(X)=\frac{C_{fg,\cA}}{2}X^2+o(X^2).
\end{equation}
\end{coro}

\begin{proof}[ of Theorem \ref{conjecturesfacteurs}]


Observe that $\mu^2(n)$ is equal to $1$ if $n$ is squarefree and is equal to $0$ otherwise. Observe also that $1+\epsilon\lambda(n)$ is equal to $2$ if $\lambda(n)=\epsilon$ and is equal to $0$ otherwise. We have also for any $n\in\N^*$ that \[\mu^2(n)\lambda(n)=\mu(n).\]


From the properties described in the last paragraph, we deduce:
\begin{align*}
2T(X)&=\!\!\sum_{\vv\in\cA(X)}\!\!\mu^2\left(\frac{g(\vv)f(\vv)}{d_{fg,\cA}}\right)\left(1+\epsilon\lambda\left(\frac{f(\vv)}{d_{f,\cA}}\right)\right)\\
&=
\!\!\sum_{\vv\in\cA(X)}\!\!\mu^2\left(\frac{g(\vv)f(\vv)}{d_{fg,\cA}}\right)
+\epsilon\!\!\sum_{\vv\in\cA(X)}\!\!\mu^2\left(\frac{g(\vv)f(\vv)}{d_{fg,\A}}\right)\lambda\left(\frac{f(\vv)}{d_{f,\cA}}\right)
\end{align*}
equals (say) $S_0(X)+\epsilon S_1(X)$. By the Squarefree conjecture for $f$ and $g$, the first sum is the main term $C_{fg,\cA}X^2+o(X^2)$. We rewrite the second sum
\[S_1(X)=\!\!\sum_{\vv\in A(X)}\!\!\lambda\left(\frac{f(\vv)}{d_{f,\cA}}\right)-\!\!\sum_{\vv\in A(X) \atop \exists p,p^2\mid f(\vv)g(\vv)d_{fg,\cA}^{-1}}\!\!\lambda\left(\frac{f(\vv)}{d_{f,\cA}}\right).\]

The first sum in the right hand side is $o(X^2)$ by Chowla's conjecture for $f$ and there remains to show the following lemma.

\begin{lm} Let $f,g$ be as in Theorem \ref{conjecturesfacteurs}, then
\[\sum_{\vv\in\cA(X);\atop\exists p,p^2\mid f(\vv)g(\vv)d_{fg}^{-1}}\!\!\lambda\left(\frac{f(\vv)}{d_{f,\cA}}\right)=o(X^2)\]
\end{lm}

(To lighten the notations, we suppose that $d_{f,\cA}=d_{fg,\cA}=1$, but it works exactly the same way in other cases.)

Let us denote the sum in the lemma by \[L(X)=\sum_{\vv\in\cA(X);\atop\exists p,p^2\mid f(\vv)g(\vv)}\!\!\lambda\left(f(\vv)\right)\]
We want to show that there exists a constant $c_1$ such that for a given $Z>0$ we have 
$\forall\epsilon>0$, $\exists X_o=X_o(\epsilon,Z)\text{ such that for }X\geq X_o$
\[\left|L(X)\right|\leq \left(\frac{c_1}{Z}+\epsilon\right) X^2.\]

This means that \[\limsup \frac{\vert L(X)\vert}{X^2}\leq \frac{c_1}{Z}+\epsilon.\]

Thus, we have \[\left|L(X)\right|=o(X^2).\]


Fix a (large) real number $Z$. 
We split the sum $F(X)$ in three parts: 

\[L(X)=\sum_{\vv\in \cA(X), \atop \exists p>X, p^2\mid f(\vv)g(\vv)}{\lambda( f(\vv))}+\sum_{\vv\in \cA(X), \atop \exists p\in[Z,X], p^2\mid f(\vv)g(\vv)}{\lambda( f(\vv))}+\sum_{\vv\in \cA(X), \atop \exists p<Z, p^2\mid f(\vv)g(\vv)}{\lambda( f(\vv))},\]
and let us denote the three sums in the right hand side by respectively $L_1(X)$, $L_2(X)$ and $L_3(X)$.
Since $\vert\lambda(n)\vert=1$, we can use the Squarefree conjecture for $fg$\footnote{We use here the (equivalent) statement given by Conjecture \ref{sfc2}}:
\[\vert L_1(X)\vert
\leq \sharp\{\vv\in \cA(X) \ \vert \ \exists p>X, p^2\mid f(\vv)g(\vv)\}=o(X^2),\]
and we use 
 Proposition \ref{1.13} to bound
\[\vert L_2(X)\vert\leq \sharp\{\vv\in \cA(X) \ \vert \ \exists p\in[Z,X], p^2\mid f(\vv)g(\vv)
\}
\ll \!\!\sum_{p\in[Z,X]}\!\! \frac{X^2}{p^2}+X\ll \frac{X^2}{Z}+\frac{X^2}{\log X}.\]

To bound the last sum $L_3(X)$, we want to consider the "small" $p$ (i.e. $p\leq Z$) such that $p^2$ divides $f(\vv)g(\vv)
$. Observe that there are a finite number of them. Put together the $\vv=(u,v)$ according to the congruence sublattices  $L_{p,i}$ defined by $(u,v)\equiv(a_i,b_i)\mod p^2$ (where $(fg)(a_i,b_i)\equiv 0\mod p$)
 and write this sublattice $L_{p.i}=\{\vv=\phi_i(\ww) \ \vert \ \ww\in\bZ^2\}$. 

Write also $f(\phi_i(\ww))=p^2h_{p,i}(\ww)$ where $h_{p,i}\in\Z[U,V]$ is a primitive squarefree homogeneous polynomial. 

When we write Chowla's conjecture for $h_{p,i}$ we obtain
\[\sum_{\vv\in L_{p,i}(X), }\lambda(f(\vv))=\sum_{\ww\in L_{p,i}(X)}\lambda(g_{p,i}(\ww))=o(X^2)\]
and, taking the sum of the finite number of sublattices :
\[\sum_{\vv\in\cA(X), \atop p^2|\,(fg)(\vv)}\lambda(f(\vv))=o(X^2).\]
This means that, for a given $\epsilon$, we have, for $X>X_0(p,\epsilon)$
\[\left|\sum_{\vv\in\cA(X), \atop p^2|\,(fg)(\vv)}\lambda(f(\vv))\right|\leq\epsilon X^2.\]
We can generalise this argument to $q=p_1\dots p_r$ a product of primes and obtain that
for $X>X_0(q,\epsilon)$
\[\left|\sum_{\vv\in\cA(X), \atop q^2|\,(fg)(\vv)}\lambda(f(\vv))\right|\leq\epsilon X^2.\]
We proceed the following way: let us denote $p_1<p_2<\dots <p_r$ the prime numbers $<Z$ and, for each $J\subset[1,r]$ put $q_J:=\prod_{j\in J}p_j$.

Applying the inclusion-exclusion principle, we see that
\[\sum_{\vv\in\cA(X), \atop \exists p<Z, p^2|\,(fg)(\vv)}\lambda(f(\vv))=-\sum_{J\not=\emptyset}(-1)^{|J|}\sum_{\vv\in\cA(X), \atop q_J^2|\,(fg)(\vv)}\lambda(f(\vv)).\]

This gives the inequality:
\[\left\vert\sum_{\vv\in\cA(X), \atop \exists p<Z, p^2|\,(fg)(\vv)}\lambda(f(\vv))\right\vert\leq\sum_{J\not=\emptyset}\left\vert\sum_{\vv\in\cA(X), \atop q_J^2|\,(fg)(\vv)}\lambda(f(\vv))\right\vert.\]

The exterior sum has about $2^{\frac{Z}{\log Z}}$ terms. As for the terms of the interior sum, we just showed that they are such that
\[\left\vert\sum_{\vv\in\cA(X), \atop q_J^2|\,(fg)(\vv)}\lambda(f(\vv))\right\vert\leq \epsilon X^2,\] for $X\geq X_0(p,\epsilon)$.

This allows to conclude that the sum is bounded by $\left(2^{\frac{Z}{\log Z}}\epsilon\right)X^2$, (or simply by $\epsilon X^2$, since $Z$ is a constant), as soon as $X$ is large enough (this "large enough" depending on $\epsilon$ and $Z$).

We showed that there exists a constant $c_1$ such that for given $Z>0$, we have $\forall\epsilon>0,\exists X_o=X_o(\epsilon,Z)\text{ such that for }X\geq X_o$,
\[\left|L(X)\right|\leq (\frac{c_1}{Z}+\epsilon) X^2.\]

This means that \[\limsup \frac{\vert L(X)\vert}{X^2}\leq \frac{c_1}{Z}+\epsilon.\]

Hence \[\left|\sum_{\vv\in\cA(X);\atop \exists p, p^2\mid (fg)(\vv)}\lambda\left(f(\vv)\right)\right|=X^2.\]


\end{proof}

\subsubsection{A "Squarefree-Liouville" Sieve}

\begin{coro}\label{sflcrible}
Let $f(U,V),g(U,V)\in\Z[U,V]$ be homogeneous polynomials in two variables. Assume that they are coprime, that no square of a nonunit in $\Z[U,V]$ divides $f(U,V)$ or $g(U,V)$, that every irreducible factor of $g$ has degree $\leq6$, and  either that $\deg f\leq3$ or that $f$ is a product of linear factors.

Let $R(U,V)\in\Z[U,V]$ be a homogeneous polynomial and $R= \gamma\cdot\prod_{0\leq i\leq r}{R_i}$ it decomposition in primitive factors. Suppose $sgn(\gamma)=+1$.
Fix
\begin{itemize}
\item a sequence $S=(p_1,\dots,p_s)$ of distinct prime numbers and
\item a sequence $T=(t_1,\dots,t_{s},t'_1,\dots,t'_{s})$ of nonnegative integers.
\end{itemize}
Let $N$ be an integer such that $p_1^{t_1+t'_1+1}\dots p_s^{t_s+t'_s+1}\mid N$ and that $p^2\mid N$ for all primes $p<(\deg f+\deg g)$. 

Suppose that there exist integers $a$, $b$ such that 
\begin{enumerate}
\item\label{paszero} $f(a,b)g(a,b)\not\equiv0\mod p^2\text{, whenever }p\mid N \text{ and }p\not=p_i\text{ for any }i,$
\item\label{memesecteur} for every $1\leq i\leq r$, one has $R_i(u,v)> 0$,
\item\label{puissanceok} such that
$v_{p_i}(f(a,b))=t_i$ and $v_{p_i}(g(a,b))=t'_i\text{ for every }i=1,\dots,s$
\item and such that
$\lambda(f(a,b))=\epsilon.$
\end{enumerate}
Then there are infinitely many pairs of integers $(u,v)$ such that
\begin{enumerate}[1.]
\item $u\equiv a\mod N, v\equiv b\mod N,$
\item for every $i$, one has $R_i(u,v)> 0$,
\item such that
$\lambda(f(u,v))=\epsilon,$
\item and such that
$f(u,v)=p_1^{t_1}\dots p_s^{t_s}\cdot l,$ and $g(u,v)=p_1^{t'_1}\dots p_s^{t'_s}l'$
where $l$ and $l'$ are squarefree and $v_{p_i}(l)=v_{p_i}(l')=0$ for all $i=1,\cdots,s$.
\end{enumerate}

\end{coro}

\begin{proof}
Fix $\epsilon\in\{-1,+1\}$. Let $\mathscr{A}=\{(u,v)\in\Z\times\Z\ \vert \ u,v\equiv a,b\mod N\}$. By Theorem \ref{conjecturesfacteurs} on $F$, $G$ and $\cA$ together with Remark \ref{thmsurlessecteurs}, there are infinitely many pairs of integers $(u,v)$ such that 
\begin{equation*}
u\equiv a\mod N,\quad v\equiv b\mod N,\end{equation*} 
\begin{equation*}
f(u,v)g(u,v)\not\equiv 0\mod p^2\text{ whenever }p\nmid N,\end{equation*} 
\begin{equation*}
\forall 0\leq j\leq n, \ R_i(u,v)>0\end{equation*}
\begin{equation*}\text{ and } \lambda(f(u,v))=\epsilon.
\end{equation*}
Condition \ref{paszero} then guarantees that $f(u,v)$ is not divisible by the square of any prime outside the sequence $S$. We also have \[u\equiv a\mod p_i^{t_i+1},\quad v\equiv b\mod p_i^{t_i+1},\quad\text{ for all }i,\]
because $p_i^{t_i+1}\mid N$ for all $i$, and hence
\[f(u,v)=f(a,b)\mod p_i^{t_i+1}\quad\text{ and }g(u,v)=g(a,b)\mod p_i^{t'_i+1}\quad\text{ for all }i.\]
Using condition \ref{puissanceok}, we conclude that
\[v_{p_i}(f(u,v))=t_i\quad\text{ and }v_{p_i}(g(u,v))=t'_i\]

\end{proof}
\section{A formula for the global root number}\label{sectionformuledusigne}

We will use the following proposition, due to Rohrlich, which gives a formula for the local root number of an elliptic curve at primes $p\geq 5$.

\begin{propo}\label{Rohrlichroot} (\cite[Proposition 2]{Rohr}) Let $p\geq5$ be a rational prime, and let $E/\Q_p$ be an elliptic curve given by the Weierstrass equation \[E:y^2=x^3-27c_4x-54c_6,\]
where $c_4,c_6\in\Z$. 
Then
\begin{displaymath}
W_p(E)=\begin{cases}
1 & \text{if the reduction of $E$ at $p$ has type $I_0$};\\
(\frac{-1}{p}) & \text{if the reduction has type $II$, $II^*$, $I_m^*$ or $I_0^*$};\\
(\frac{-2}{p}) & \text{if the reduction has type $III$ or $III^*$};\\
(\frac{-3}{p}) & \text{if the reduction has type $IV$ or $IV^*$;}\\
-(\frac{-c_6}{p}) & \text{if the reduction has type $I_{m}$};\\
\end{cases}
\end{displaymath}
\end{propo}

Note that the local root number at infinity is always $W_\infty(E)=-1$.
The values in the previous proposition are Jacobi symbols. 

\subsection{Notations and definitions}


For each rational number $t\in\Q$ there is a unique pair of coprime integers \begin{equation}\label{t=xy}(u(t),v(t))\in\Z\times\Z_{>0}\quad\text{ such that }t=\frac{u(t)}{v(t)}.\end{equation} To simplify the notation, we will simply write it as $(u,v)$ if $t$ is obvious.

For an integer $N$ and a prime number $p$, we denote by $N_{(p)}$ the integer such that $N=p^{\nu_p(N)}N_{(p)}$. Similarly, given a positive integer $\delta$, we will denote by $a_{(\delta)}$ the positive integer such that \[a_{(\delta)}=\frac{ a}{\prod_{p\mid\delta}{p^{\nu_p(a)}}}.\]
\subsubsection{A curve isomorphic to a fiber}\label{notaroot}

Let $\E$ be an elliptic surface described by the Weierstrass equation \[\E:y^2=x^3-27c_4(T)x-54c_6(T),\]
that we suppose to be a minimal Weierstrass model.

If $t\in\Q$ is not an integer, it is very likely that $c_6(t)$ is not an integer either, and in this case, we cannot directly use Proposition \ref{Rohrlichroot} on $\E_t$ to find its root number. However, we can consider the elliptic curve which is isomorphic to $\E_t$: \[\E_{u,v}:y^2=x^3-27v^{4k}c_4(u/v)x-54v^{6k}c_6(u/v),\]
where $(u,v)$ are as defined as in (\ref{t=xy}) and $k$ is the smallest integer such that both $4k\geq\deg A$ and $6k\geq\deg B$. To lighten the notation we will denote \[\overline{c_4}(u,v)=v^{4k}c_4(u/v),\quad\text{ and }\overline{c_6}(u,v)=v^{6k}c_6(u/v).\]
The coefficients of $\E_{u,v}$ are integers and we can apply Proposition \ref{Rohrlichroot}. Since $\E_{u,v}$ and $\E_t$ are isomorphic, we have \[W(\E_t)=W(\E_{u,v}).\]

Note that this elliptic curve's discriminant is equal to $\Delta_\E(u,v)$ the discriminant of the elliptic surface evaluated at $(u,v)$: \[\Delta_\E(u,v)=\Delta(\E_{u,v})=v^{12k}(c_4(u/v)^3-c_6(u/v)^2).\]

\subsubsection{Local constancy}

\begin{lm} \label{localrootlemma}
The local root number $W_p$ is locally constant for the $p$-adic topology.

In other words, if we denote $U=U_\E\subset\P^1$ the affine open subset above which $\E$ has good reduction, the map from $U(\Q_p)$ to $\{\pm1\}$ defined by
\[\overline{t}\mapsto W_p(\E_t)\]
 is continuous (i.e. locally constant).
\end{lm}

For functions over  $\Z\times\Z$ or over $\Q$, we use the following similar properties:

\begin{definition}\label{localconstancy}
Let $\mathscr{S}\subseteq\Z\times\Z$. A function $f:\mathscr{S}\rightarrow\{\pm1\}$ is said \emph{locally constant} if there exists $N\in\mathbb{N}^*$ such that the congruences $(u,v)\equiv(u',v')\mod N$ implies
\[f\left(u,v\right)=f\left(u',v'\right).\] 

Similarly, let $\mathscr{T}\subseteq\Q$. A function $\varphi:\mathscr{T}\rightarrow\{\pm1\}$ is \emph{locally constant} if there exists $N\in\mathbb{N}^*$ such that the congruences $(u,v)\equiv(u',v')\mod N$ implies
\[\varphi\left(\frac{u}{v}\right)=\varphi\left(\frac{u'}{v'}\right).\] 
\end{definition}

However, the global root number is not locally constant. More precisely, it is never locally constant on non-isotrivial elliptic surfaces, as we will see later in this section, and is locally constant on the connected components of $$\R^2-\{(u,v)\in\R^2 \ \vert \ \tilde{F}(u,v)=0\}$$ on families of quartic or sextic twists by $F(T)\in\Z[T]$ i.e. on isotrivial elliptic surfaces with $j$-invariant outside of $0$ and $1728$, as we will see in a forthcoming paper \cite{Desjardins4}.\\

\subsection{Monodromy of reduction type}\label{monodromie}

Let $\E$ be the elliptic surface with discriminant $\Delta_\E(U,V)$ described by the minimal Weierstrass equation
\[Y^2=X^3-27c_4(T)X-54c_6(T).\]

Given a place over $\Q[T]$ corresponding to a monic irreducible homogeneous polynomial $P(U,V)\in\Z[U,V]$, we would like to know what is the type of the fiber $\E_t$ of $\E$ above $t\in\Q$ at a prime number $p\mid P(u,v)$.

Let $(u,v)\in\Z\times\Z_{>0}$ be the integers related to $t$ as defined in Section \ref{t=xy}.
To solve this question, we will analyze the values of the $p$-adic valuation of the discriminant $\Delta_\E(u,v)$ and of the coefficients of the curve $\E_{u,v}$

We suppose that $p$ has the following properties:
\begin{enumerate}
\item $p\not=2,3$;
\item the numerators $n_4$, $n_6$ and $n_\Delta$, of the contents of the polynomials $c_4(T)$, $c_6(T)$ and $\Delta(T)$ are not divisible by $p$\footnote{Since the Weierstrass equation is supposed minimal, we have that $c_4\in\left(\frac{1}{3^3}\right)\Z[T]$ and $c_6\in\left(\frac{1}{2\cdot3^3}\right)\Z[T]$. We have thus as well that $\Delta\in\left(\frac{1}{2\cdot3^3}\right)\Z[T]$.};
\item if $p\mid P(u,v)$, then for all $P'\not=P$ of bad reduction, one has $p\nmid P'(u,v)$. In other words, if we denote by $\Res(P,P')$ the resultant of two polynomials, we have \[p\nmid\prod_{Q,Q'\in\mathscr{B}; Q\not=Q'}{\mathrm{Res}(Q,Q')}.\]
\end{enumerate}
Note that almost all prime number verify the properties 1, 2, and 3.
There is only a finite number of exceptions, and these are the prime numbers dividing the integer

\[\delta=2\cdot3\cdot n_4\cdot n_6\cdot n_\Delta\prod_{Q,Q'\in\mathscr{B}\text{, }Q\not=Q'}{\mathrm{Res}(Q,Q'}),\]
where $Q,Q'$ run through polynomials associated to generic places of bad reduction.

The local root number at $p$ depends on the values of $\nu_p(\overline{c_4}(u,v))$, $\nu_p(\overline{c_6}(u,v))$ and $\nu_p(\Delta_\E(u,v))$ as seen in the Néron-Kodaira classification of the singular fibers (see \cite{Kodaira}, \cite{Neronfibre}). When $p$ respects the properties (1) to (3), we have the following equalities
\[\nu_p(\overline{c_4}(u,v))=\mathrm{w}_P(\overline{c_4}(U,V))\cdot \nu_p(P(u,v)),\] 
\[\nu_p(\overline{c_6}(u,v))=\mathrm{w}_P(\overline{c_6}(U,V))\cdot \nu_p(P(u,v))\]
\[\nu_p(\Delta_\E(u,v))=\mathrm{w}_P(\Delta_\E(U,V))\cdot \nu_p(P(u,v)),\]
where $\mathrm{w}_P$ is the place of $\Q[T]$ associated to $P$. The reduction type at such a $p$ depends only on the $p$-adic valuation of $P(u,v)$.

This argument leads to the following lemma which is summarized in Table \ref{tableaumonodromie} below:

\begin{lm}\label{lemmemonodromie}
Let $\E$ be an elliptic surface, and $p$ be a prime number that does not divide $\delta$.
Let $P$ be the monic irreducible homogeneous polynomial associated to a place of $\Q[T]$. For all $t\in\Q$, let $(u,v)\in\Z\times\Z_{>0}$ be the pair of coprime integers such that $t=\frac{u}{v}$ as in Section \ref{t=xy}. Suppose that $P(u,v)$ is divisible by $p$ and put $n=\nu_p(P(u,v))$.

Then :
\begin{enumerate}
\item If $\E$ has type $I_m$ at $P$, then $\E_{t}$ has type $I_{nm}$ at $p$.
\item If $\E$ has type $I_m^*$ at $P$, then $\E_{t}$ has type $I_{nm}$ at $p$ If $n$ is even, and of type $I_{nm}^*$ if $n$ is odd.
\item If $\E$ has type $II$ at $P$, then $\E_{t}$ has type $I_0$, $II$, $IV$, $I_0^*$, $IV^*$, $II^*$ at $p$ if $n\equiv0,1,2,3,4,5\mod6$ respectively.
\item If $\E$ has type $II^*$ at $P$, then $\E_{t}$ has type $I_0$, $II^*$, $IV^*$, $I_0^*$, $IV$, $II$ at $p$ if $n\equiv0,1,2,3,4,5\mod6$ respectively.
\item If $\E$ has type $III$ at $P$, then $\E_{t}$ has type $I_0$, $III$, $I_0^*$, $III^*$ at $p$ if $n\equiv0,1,2,3\mod4$ respectively.
\item If $\E$ has type $III^*$ at $P$, then $\E_{t}$ has type $I_0$, $III^*$, $I_0^*$, $III$ at $p$ if $n\equiv0,1,2,3\mod4$ respectively.
\item If $\E$ has type $IV$ at $P$, then $\E_{t}$ has type $I_0$, $IV$, $IV^*$ at $p$ if $n\equiv0,1,2\mod3$ respectively.
\item If $\E$ has type $IV^*$ at $P$, then $\E_{t}$ has type $I_0$, $IV^*$, $IV$ at $p$ if $n\equiv0,1,2\mod3$ respectively.
\end{enumerate}
\end{lm}

\begin{center}
\begin{table}
   \begin{tabular}{| c c l | c c l |}
     \hline
Type of $\E_T$  & $n=\nu_p(P(u,v))$  & Type of $\E_t$ & Type of $\E_T$ & $n=\nu_p(P(u,v))$ & Type of $\E_t$ \\
at $P(T)$& & at $p$& at $P(T)$ && at $p$\\
\hline
     \multirow{2}{*}{$I_m$} & \multirow{2}{*}{$n\geq 1$} & \multirow{2}{*}{$I_{mn}$} & \multirow{2}{*}{$I_m^*$} & $0\mod2$ & $I_{mn}$ \\
&&&& $1\mod2$ &$I_{mn}^*$  \\
\hline
\multirow{6}{*}{$II$} &  $0\mod6$ & $I_0$ &\multirow{6}{*}{$II^*$}& $0\mod6$ & $I_0$\\
& $1\mod6$ & $II$ &&$1\mod6$ &$II^*$\\
&$2\mod6$ & $IV$ && $2\mod6$ &$IV^*$\\
&$3\mod6$ & $I_0^*$& &$3\mod6$&$I_0^*$\\
&$4\mod6$ & $IV^*$ & &$4\mod6$& $IV$\\
&$5\mod6$	& $II^*$ & &$5\mod6$&$II$\\
\hline
 \multirow{4}{*}{$III$}    & $0\mod2$& 	$I_0$ & \multirow{4}{*}{$III^*$}& $0\mod4$ & $I_0$\\
&$1\mod4$&$III$ & &$1\mod4$ & $III^*$ \\
&$2\mod4$& $I_0^*$ && $2\mod4$ & $I_0^*$\\
&$3\mod4$&	$III^*$&& $3\mod4$	& $III$	\\
\hline
 \multirow{3}{*}{$IV$}    & $0\mod3$ & $I_0$ & \multirow{3}{*}{$IV^*$} &$0\mod3$ & $I_0$\\
&$1\mod3$ & $IV$&& $1\mod3$&$IV^*$\\
&$2\mod3$	& $IV^*$	&& $2\mod3$& $IV$\\
\hline
\multirow{2}{*}{$I_0^*$} & $0\mod2$ & $I_0$ & \multirow{2}{*}{$I_0$} & \multirow{2}{*}{$n\geq0$} & \multirow{2}{*}{$I_0$}\\ & $1\mod2$ & $I_0^*$ & 			& & \\
   \hline
\end{tabular}
\caption{\label{tableaumonodromie}Lemma \ref{lemmemonodromie} 
}
\end{table}
 \end{center}
\begin{proof}
This can be directly deduced from Tate's algorithm \cite{Tatefibre}.
\end{proof}

\subsection{Decomposition of the root number according to the generic places}

Let $\E$ be an elliptic surface described by the Weierstrass equation: \[\E:y^2=x^3-27c_4(T)x-54c_6(T),\]
that we suppose to be a minimal Weierstrass model.

When $t\in\Q$ is not an integer and $\E_t$ is non-singular, then as in Section \ref{notaroot}, the root number of $\E_t$ is the same as the isomorphic elliptic curve $\E_{u,v}$ and thus $W(\E_t)=W(\E_{u,v})$.
Remember that $\E_{u,v}$ is the curve given by the equation:
\[\E_{u,v}:y^2=x^3-27\overline{c_4}(u,v)x-54\overline{c_6}(u,v),\]
where
\[\overline{c_4}(u,v)=v^{4k}c_4(u/v) \qquad \overline{c_6}(u,v)=v^{6k}c_6(u/v)\]
and
where $k$ be the smallest integer such that both $4k\geq\deg A$ and $6k\geq\deg B$ holds. This curve has discriminant $\Delta(\E_{u,v})=v^{12k}\Delta(\E_{u/v})$.\\
We have thus 
\begin{align*}
W(\E_t)&= W(\E_{u,v})\\
& = \prod_{p\leq\infty}{W_p(\E_{u,v})}\\
& = -\prod_{p\mid\Delta(\E_{u,v})}{W_p(\E_{u,v}})\\
\end{align*}

Let us write
\begin{displaymath}
\overline{c_4}(U,V)=\frac{n_4}{d_4}\prod_{P\mid c_4}{P(U,V)^{\mathrm{w}_P (c_4)}},
\end{displaymath}
\begin{displaymath}
\overline{c_6}(U,V)=\frac{n_6}{d_6}\prod_{P\mid c_6}{P(U,V)^{\mathrm{w}_P(c_6)}}, 
\end{displaymath}
\begin{displaymath}
\Delta_\E(U,V)=\frac{n_\Delta}{d_\Delta}\prod_{P\vert\Delta}{P(U,V)^{\mathrm{w}_P(\Delta)}},
\end{displaymath}
with minimal $n_4,n_6,n_\Delta,d_4,d_6,d_\Delta\in\mathbb{Z}$ and where $\mathrm{w}_P$ is the valuation at $P$. Remark that the prime factors of the denominators ($d_4$, $d_6$ and $d_\Delta$) will only be $2$ or $3$ since $27c_4(T),54c_6(T)\in\Z[T]$.

As we saw in Section \ref{monodromie}, we have to treat differently the local root number at 2 and 3, in $d_i$ for $i=0,\dots,5$ and at $p$ that divide both $P$ and $P'$, that is to say the $p$ dividing \[\prod_{P,P'\mid B_\E\atop P\not= P'}{\mathrm{Res}(P,P')}.\] 
For the other primes, the behavior of the reduction type of the fibers is described by the lemma \ref{lemmemonodromie}.
Put
\begin{displaymath}
\delta=2\cdot3\cdot n_4\cdot n_6\cdot n_\Delta \prod_{P,P'\mid\Delta_\E}{\mathrm{Res}(P, P')}.
\end{displaymath}
The root number can be expressed as
\begin{displaymath}
W(\mathscr{E}_t)=-\prod_{p\mid\delta}{W_p(\E_t)}\prod_{P\mid\Delta_\E}{\mathscr{W}_{P}(u,v)},
\end{displaymath}
where for each $P$ primitive factor of $\Delta_\E$:
\begin{displaymath}
\W_{P}(u,v)=\prod_{p \nmid\delta:p\mid P(u,v)}{W_p(\mathscr{E}_t)}.
\end{displaymath}

Now, by Lemma \ref{monodromie}, each of the $W_p(\E_t)$ depends only of $P$ and of $n=\nu_p(P(u,v))$. In particular $n=1$, then the type of reduction of $\E_t$ is the same as the one of $\E$ at $P$. We have thus in this case:
\[W_p(\E_t)=\begin{cases}
\left(\frac{-\varepsilon_{P}}{p}\right)&\text{ if $\E$ has additive reduction at $P$}\\
-\left(\frac{-\overline{c_6}(u,v)}{p}\right)&\text{ if $\E$ has multiplicative reduction at $P$,}
\end{cases}\]
where \[\varepsilon_{P}=\begin{cases}
-1&\text{ if the type is $II$, $II^*$ or $I_m^*$ ($m\leq0$)}\\
-2&\text{ if the type is $III$ or $III^*$}\\
-3&\text{ if the type is $IV$ or $IV^*$.}\\
\end{cases}\]

However, when $n\geq2$, the type of $\E$ is suceptible to change. For this reason, we introduce a corrective function, denoted by $h_{P}$ and equal to $\frac{g_P}{\W_P}$ so we can write:

\[\W_{P}(u,v)=h_{P}(u,v)\begin{cases}
\left(\frac{-\varepsilon_{P}}{P(u,v)}\right)_\delta&\text{ if $P$ additive type}\\
(-1)^{\omega(P(u,v)_{(\delta)})}\left(\frac{-\overline{c_6}(u,v)}{P(u,v)}\right)_\delta&\text{ if $P$ has multiplicative type,}
\end{cases}\]
where $\omega(n)$ is the number of prime factors of the integer $n$ and $(\frac{\cdot}{\cdot})_\delta$ is the quadratic symbol defined as follows:

\begin{definition}
For each pair of integers $(a,b)\in\Z\times\Z$ and even integer $\delta$,

\begin{equation}\label{quadraticsymboldef}\left(\frac{a}{b}\right)_\delta=\prod_{p\nmid\delta}{\left(\frac{a_{(p)}}{p}\right)^{\nu_p(b)}}\end{equation}
where the product runs through the prime number $p\nmid\delta$ and $(\frac{\cdot}{p})$ is Legendre symbol.
\end{definition}

If $a,b,\delta$ are two-by-two coprime, then the symbol $(\frac{a}{b})_\delta$ is the classical Jacobi symbol. The properties of this symbol are displayed in Appendix \ref{quadraticsymbol}.

Note that \begin{align}(-1)^{\omega(P(u,v)_{(\delta)})}&=\prod_{p\nmid\delta \atop p\mid P(u,v)_{(\delta)}}{(-1)^{v_p(P(u,v))}}.\\
&=\lambda(P(u,v))\prod_{p\mid\delta}{(-1)^{\nu_p(P(u,v))}}\end{align}

Thus, we have the following:

\begin{thm} \label{formuledusigne}

Let $\mathscr{E}$ be an elliptic surface over $\mathbb{Q}$. Let $\delta$ be defined as previously. For $t\in\Q$, let $(u,v)\in\Z\times\Z_{>0}$ be as in Section \ref{t=xy} the pair of coprime integers such that $t=\frac{u}{v}$.

Then, the root number can be written as \[W(\E_t)=\lambda(M_\E(u,v))\cdot\prod_{p\mid\delta}{W_p(\E_{u,v})}\cdot\prod_{P\mid\Delta_\E}{h_{P}(u,v)g_{P}(u,v)}\]
 where the functions $h_{P}$ are the corrective functions defined earlier and whose formula is given in Table \ref{contributions} and
\[ g_{P}(u,v)=
\begin{cases}
\left(\frac{-\varepsilon_{P}}{P(u,v)}\right)_{\delta}&\text{ if $\E$ has additive reduction}\\
\left(\frac{-\overline{c_6}(u,v)}{P(u,v)}\right)_{\delta}\left(\prod_{p\mid\delta}{(-1)^{\nu_p(P(u,v))}}\right)&\text{ if $\E$ has multiplicative reduction,}
\end{cases}.
\]

\end{thm}

\begin{center}
\begin{table}
   \begin{tabular}{ | l ||  c |}
     \hline
Type & $h_{P}(u,v)$ \\
\hline
$I_0$  & 1 \\
\hline
$I_0^*$ 		& 1 \\
\hline
$II$, $II^*$	& $\prod \limits_{\underset{p^2\mid P(u,v)}{p\nmid\delta}} 
\begin{cases}
\left(\frac{-3}{p}\right) & \nu_p(P(u,v))\equiv2,4\mod6 \\
+1 & \text{otherwise.}
\end{cases}$ \\
     \hline
$III$, $III^*$	& $\prod \limits_{\underset{p^2\vert P(u,v)}{p\nmid\delta}} 
\begin{cases}
\left(\frac{-1}{p}\right) & \nu_p(P(u,v))\equiv2\mod4\\
+1 & \text{otherwise.}
\end{cases}$ \\
\hline
$IV$, $IV^*$	& $\prod \limits_{\underset{p^2\vert P(u,v)}{p\nmid\delta}} 
\begin{cases}
\left(\frac{-3}{p}\right) & \nu_p(P(u,v))\equiv2,3,4\mod6\\
+1 & \text{otherwise.}
\end{cases}$ \\

\hline
$I_m^*$  ($m\geq1$) 	& $\prod \limits_{\underset{p^2\vert P(u,v)}{p\nmid\delta}} \left(-\left(\frac{-\overline{c_6}(u,v)_{(p)}}{p}\right)\right)^{\nu_p(P(u,v))-1}$ \\
\hline
$I_m$ ($m\geq1$)  	& $\prod \limits_{\underset{p^2\vert P(u,v)}{p\nmid\delta}} \left(-\left(\frac{-\overline{c_6}(u,v)_{(p)}}{p}\right)\right)^{\nu_p(P(u,v))-1}$ \\ 
 &\\ \hline
\end{tabular}
\caption{\label{contributions}Corrective functions $h_P$ of Theorem \ref{formuledusigne} 
}
\end{table}
 \end{center}

\begin{proof}
%
To complete the proof, we simply need to find the expression of the corrective functions $h_{P}$.
We compute it according to the reduction of $\mathscr{E}_t$ at $p$ such that $p\nmid\delta$ and $p\mid P(u,v)$ and the monodromy given by Lemma \ref{lemmemonodromie}.

\textbf{Case 1}: Suppose the reduction of $\E$ has type $I_0$ at $P$. Then, the reduction of $\E_t$ at $p$ has also type $I_0$.
In this case: 
\begin{displaymath}
\W_{P}(u,v)=\prod \limits_{\underset{p\mid P(u,v)}{p\nmid\delta}}  {W_p(\mathscr{E}_t)}=1,
\end{displaymath}
for all $(u,v)\in\Z\times\Z_{>0}$ coprime.

\textbf{Case 2}: Suppose the reduction at $P$ has type $I_0^*$. Then the reduction type at $p$ depends on the parity of $\nu_p(P(u,v))$.
In this case:
\begin{align*}
\W_{P}(u,v)&=\prod_{ p\nmid\delta \atop p\mid P(u,v)} 
\begin{cases}
1 & \text{if } \nu_p(P(u,v))\text{ even and non-zero (type $I_o$),}\\
\left(\frac{-1}{p}\right) & \text{if }\nu_p(P(u,v))\text{ odd (type $I_o^*$);}\\
\end{cases}\\
&=\prod \limits_{\underset{p\mid P(u,v)}{p\nmid\delta}} 
\left(\frac{-1}{p}\right)^{\nu_p(P(u,v))}\\
&=\left(\frac{-1}{P(u,v)}\right)_{\delta},
\end{align*}
for all $(u,v)\in\Z\times\Z_{>0}$ coprime.

\textbf{Case 3}: Suppose the reduction at $P$ has type $II$ or $II^*$, then the reduction type at $p$ depends of $\nu_p(P(u,v))\mod6$.
In this case: \begin{align*}
\W_{P}(u,v)&=\prod_{p\mid\delta \atop p\mid P(u,v)}{
\begin{cases}
\left(\frac{-1}{p}\right) &  \text{if }\nu_p(P(u,v))\equiv 1,5\mod6\text{ (type $II$ or $II^*$),}\\
\left(\frac{-3}{p}\right) &  \text{if }\nu_p(P(u,v))\equiv 2,4\mod6 \text{ (type $IV$ or $IV^*$),} \\
\left(\frac{-1}{p}\right) &  \text{if }\nu_p(P(u,v))\equiv 3\mod6\text{ (type $I_0^*$)}\\
1 & \text{if }\nu_p(P(u,v))\equiv 0\mod6\text{ (type $I_0$);}\\
\end{cases}}\\
&=\left(\frac{-1}{P(u,v)}\right)_{\delta}\cdot\prod
\limits_{p\nmid\delta \atop \nu_p(P(u,v))\equiv2,4\mod6} \left(\frac{-3}{p}\right).\\
\end{align*}
\textbf{Case 4}: Suppose the reduction at $P$ has type $III$ or $III^*$, then the reduction type at $p$ depends on $\nu_p(P(u,v))\mod4$.
In this case: \begin{align*}
\W_{P}(u,v)&=\prod \limits_{\underset{p\mid P(u,v)}{p\nmid\delta}} 
\begin{cases}
\left(\frac{-2}{p}\right) &  \text{if }\nu_p(P(u,v))\equiv 1,3\mod4\text{ (type $III$ or $III^*$),}\\
\left(\frac{-1}{p}\right) &  \text{if }\nu_p(P(u,v))\equiv 2\mod 4 \text{ (type $I_0^*$)}, \\
1 & \text{if } \nu_p(P(u,v))\equiv 0\mod6\text{ (type $I_0$)}\\
\end{cases}
\\
&=\left(\frac{-2}{P(u,v)}\right)_{\delta}\cdot\prod \limits_{\underset{\nu_p(P(u,v))\equiv2\mod4}{p\nmid\delta}} 
\left(\frac{-1}{p}\right).
\end{align*}
\textbf{Case 5}: Suppose the reduction at $P$ has type $IV$ or $IV^*$, then the type de reduction at $p$ depends on $\nu_p(P(u,v))\mod3$.
In this case:
\begin{align*}
\W_{P}(u,v)&=\prod \limits_{\underset{p\mid P(u,v)}{p\nmid\delta}} 
\begin{cases}
\left(\frac{-3}{p}\right) & \text{if } \nu_p(P(u,v))\equiv 1,2\mod3\text{ (type $IV$ or $IV^*$),}\\
1 & \text{if } \nu_p(P(u,v))\equiv 0\mod3\text{ (type $I_0$)}\\
\end{cases}
\\
&=\left(\frac{-3}{P(u,v)}\right)_{\delta}\cdot\prod \limits_{\underset{\nu_p(P(u,v))\equiv2,3,4\mod6}{p\nmid\delta}} 
\left(\frac{-1}{p}\right)
\end{align*}
\textbf{Case 6}: Suppose the reduction at $P$ has type $I_v^*$. Then the reduction type at $p$ depends on the parity of $\lambda$.
We get the formula:
\begin{align*}
\W_{P}(u,v)&=\prod \limits_{\underset{p\mid P(u,v)}{p\nmid\delta}} 
\begin{cases}
-\left(\frac{-\overline{c_6}(u,v)_{(p)}}{p}\right) &\text{if } \nu_p(P(u,v))\text{ even (type $I_m$),}\\
\left(\frac{-1}{p}\right) & \text{if }\nu_p(P(u,v))\text{ odd (type $I_m^*$)}\\
\end{cases}
\\
&=\left(\frac{-1}{P(u,v)}\right)_{\delta}\cdot\prod \limits_{\underset{p^2\vert P(u,v)}{p\nmid\delta}} 
\left(-\left(\frac{-\overline{c_6}(u,v)_{(p)}}{p}\right)\right)^{\nu_p(P(u,v))-1}.\end{align*}
\textbf{Case 7}: Suppose the reduction at $P$ is multiplicative, then by a similar reasoning, the reduction at $p$ is also multiplicative.
In this case:
\begin{align*}
\W_{P}(u,v)=&\prod \limits_{\underset{p\mid P(u,v)}{p\nmid\delta}} 
-\left(\frac{-\overline{c_6}(u,v)_{(p)}}{p}\right)\nonumber\\
= &\prod \limits_{\underset{p\mid P(u,v)}{p\nmid\delta}} 
-\left(\frac{-\overline{c_6}(u,v)_{(p)}}{p}\right)^{\nu_p(P(u,v))}\cdot\prod \limits_{\underset{p^2\vert P(u,v)}{p\nmid\delta}} 
-\left(\frac{-\overline{c_6}(u,v)_{(p)}}{p}\right)^{\nu_p(P(u,v))-1}\nonumber\\
=&\left(\frac{-\overline{c_6}(u,v)}{P(u,v)}\right)_{\delta} \cdot \prod \limits_{\underset{p\mid P(u,v)}{p\nmid\delta}} 
(-1)^{\nu_p(P(u,v))} 
\cdot\prod \limits_{\underset{p^2\vert P(u,v)}{p\nmid\delta}} 
-\left(\frac{-\overline{c_6}(u,v)_{(p)}}{p}\right)^{\nu_p(P(u,v))-1} \nonumber\\
=&\lambda(P(u,v))\left(\prod_{p\mid\delta}{(-1)^{\nu_p(P(u,v))}}\right)\cdot\left(\frac{-\overline{c_6}(u,v)}{P(u,v)}\right)_{\delta}
\cdot\prod \limits_{\underset{p^2\vert P(u,v)}{p\nmid\delta}} 
-\left(\frac{-\overline{c_6}(u,v)_{(p)}}{p}\right)^{\nu_p(P(u,v))-1},\nonumber\\
\end{align*}
for any pair $(u,v)\in\mathbb{Z}\times\Z_{>0}$ of coprime integers.

This gives us the decomposition of the theorem.
\end{proof}

\section{Variation of the different components of the root number}\label{sectionvariation}

\subsection{The function $\prod_{p\mid\delta}{W_p(\E_t)}\prod_{P\mid\Delta}{g_{P}}(u,v)$}\label{constancelocaleg}

\begin{propo}\label{existencedunM}
Let $\E$ be an elliptic surface. Let $\delta$ be the integer defined in section \ref{monodromie}. 

Then, there exists an integer $N_\E\in\N^*$ and a non-zero homogeneous polynomial $R_\E\in\Z[U,V]$ such that the function $\varphi_\E:\Q\rightarrow\{-1,+1\}$ defined as \[t=\frac{u}{v}\mapsto \varphi_{\E}(t):=\prod_{p\mid\delta}{W_p(\E_{t})}\prod_{P\in\mathscr{B}}{g_{P}(u,v)}\] is such that $\varphi_\E(t)=\varphi_\E(t')$ for every $t,t'\in\Q$ such that the associated pairs of coprime integers $(u,v),(u',v')\in\Z\times\Z_{>0}$ (as in Section \ref{t=xy}) are such that $(u,v)\equiv (u',v')\mod N_\E$ and such that for every irreducible factor $R_i$ of $R$, one has $R_i(u,v)$ of a given sign $\epsilon_i\in\{\pm1\}$.

Thus $\varphi_\E$ is locally constant\footnote{In the sense of Definition \ref{localconstancy}.} on connected components of $\R^2-\{(u,v)\in\R^2 \ \vert \ R(u,v)=0\}$.
\end{propo}

\begin{proof}
Let $\E$ be an elliptic surface and let $P\in\mathscr{B}$ be a polynomial associated to a place of bad reduction on $\E$.

When the reduction of $\E$ at $P$ is not multiplicative,
we have \[g_{P}(u,v)=\Big(\frac{\varepsilon_P}{P(u,v)}\Big)_\delta\] where \[\varepsilon_P=\begin{cases}-1&\text{if $P$ has type $II$, $II^*$, $I_0^*$ or $I_m^*$};\\
-2&\text{if $P$ has type $III$, $III^*$};\\
-3&\text{if $P$ has type $IV$ or $IV^*$,}\end{cases}\]
which depends respectively of the value of $P(u,v)_{(\delta)}$ modulo 4, 8 and 12 and on the sign of $P(u,v)$.

In the case where the surface has multiplicative reduction at $P$,
we have \[g_{P}(u,v)=\left(\frac{-\overline{c_6}(u,v)}{P(u,v)}\right)_\delta\prod_{p\mid\delta}{(-1)^{\omega(P(u,v))}}.\] For those, Proposition \ref{adelicsymbol} guaranties the existence of integers $N_P$ and homogeneous polynomials $R_P(U,V)$ such that $g_{P}(u,v)=g_{P}(u',v')$ for all $(u,v),(u',v')\in\Z\times\Z$ such that $(u,v)\equiv(u',v')\mod N_P$ in a connected component of $\R^2-\{(u,v)\in\R^2 \ \vert \ R(u,v)=0\}$.
Thus, $\prod_{P\mid\Delta_\E}{g_{P}}$ is constant on the intersection of a class of $(u,v)$ modulo $24\prod_{P\in\mathscr{M}}{N_P}$ and a connected component.

We conclude the proof by observing that the local root numbers at $p\mid \delta$ of the fibres are locally constant for the $p$-adic topology (Lemma \ref{localrootlemma}). For each $p\mid\delta$, let $\alpha_p$ be the smallest integer $\alpha_p$ such that $W_p(E_t)=W_p(E_{t'})$ if  and only if $t\equiv t'\mod p^{\alpha_p}$.
The function $\varphi_\E$ is thus constant on the intersection of the double congruence classes modulo \[N_\E=24\prod_{p\mid\delta}{p^{\alpha_p}}\prod_{P\in\mathscr{M}}{N_P},\]
and a connected component of $\R^2-\{(u,v)\in\R^2 \ \vert \ R(u,v)=0\}$.
\end{proof}

\subsection{The function $h_{P}$ when the reduction at $P$ has type $II$, $II^*$, $IV$ or $IV^*$}

\begin{lm}\label{constance24}
Let $P\in\Z[U,V]$ be a polynomial associated to a place of type $II$, $II^*$ $IV$ or $IV^*$. We assume that \[\mu_3\subseteq \Q[T]/P(T,1),\] where $\mu_3$ is the group of third roots of unity.

Then for all $t\in\Q$ and the associated pair of coprime integers $(u,v)\in\Z\times\Z_{>0}$ one has $h_{P}(u,v)=+1$.
\end{lm}

\begin{proof} 

Let $P(T)\in\Z[T]$ be an irreducible polynomial, and let be $K=\Q[T]/P(T)$. Let $\mu_3$ be the group of third roots of unity, and suppose $\Q(\mu_3)\subseteq K$. Then 
, for every pair $(u,v)$ of coprime integers, there does not exist prime numbers $p\equiv2\mod3$ dividing $P(u,v)$. In other words, if $p\mid P(u,v)$ then $p\equiv 1\mod3$.

We have thus \begin{align*}
h_{P}(u,v) & =\prod \limits_{\underset{p^2\vert P(u,v)}{p\nmid\delta}} 
\begin{cases}
\Big(\frac{-3}{p}\Big) & \text{if }\nu_p(P(u,v))\equiv2,4\mod6\\
+1 & \text{otherwise.}
\end{cases}\\
&=+1.\\
\end{align*}
\end{proof}

\begin{rem}
One can easily give examples of homogeneous polynomials $P\in\Z[U,V]$ satisfying the hypothesis $\mu_3\subseteq\Q[T]/(P_i(T,1))$ for every irreducible factors $P_i$ of $P$. We have in particular those of the form \[P(U,V)=3A(U,V)^2+B(U,V)^2,\]
where $A(U,V),B(U,V)\in\Z[U,V]^*$ are coprime.
This polynomial is not necessarily irreducible in general, but for 
every irreducible factor $P_i$ and the corresponding field $K_i=\Q[T]/(P_i(T,1))=\Q(\alpha_i)$ where $\alpha_i$ is a root of $P_i$, one has $3A(\alpha_i,1)^2+B(\alpha_i,1)^2=0$, thus $-3=(B(\alpha_i)A(\alpha_i)^{-1})^2$ and this proves that $\mu_3\subset K_i$.

\end{rem}

\begin{rem}
The proof is inspired by a more general result due to Bauer (see \cite[p.548]{Neukirch}) which states that for $L/K$ a Galois extension and $M/K$ a finite extension of $L$, one has
\[P(L/K)\supseteq P(M/K)\Leftrightarrow L\subseteq M,\]
where $P(L/K):=\{p\text{ prime of } K \ \vert \ \exists\mathfrak{p}\text{ prime of $L$ of degree 1 lying over $p$}\}$.
\end{rem}

\subsection{The function $h_{P}$ when the reduction at $P$ has type $III$ or $III^*$}

\begin{lm}\label{constance3}
Let $P\in\Z[U,V]$ be an irreducible polynomial associated to a place of type $II$, $II^*$, $IV$ or $IV^*$. Assume that \[\mu_4\subseteq \Q[T]/P(T,1),\] where $\mu_4$ is the group of the fourth roots of unity.

Then for all $t\in\Q$ and its associated pair of coprime integers $(u,v)\in\Z\times\Z_{>0}$ one has $h_{P}(u,v)=+1$.
\end{lm}

\begin{proof}
Similar to Lemma \ref{constance24}.
\end{proof}

\begin{rem}
One can easily give examples of polynomials satisfying the hypothesis $\mu_4\subseteq\Q[T]/(P_i(T,1))$ for every factor $P_i$; in particular those of the form \[P(U,V)=A(U,V)^2+B(U,V)^2,\]
where $A(U,V),B(U,V)\in\Z[U,V]^*$ are coprime.
\end{rem}

\subsection{The function $h_{P}$ when the reduction at $P$ has type $I_m^*$ or $I_m$}

First, the following lemma gives a criterion to fix the values of the function $h_{P}$.

\begin{lm}\label{constancehpotmult}
Let $\E$ be an elliptic surface which admits a place of type $I_m^*$ or $I_m$ 
whose associated polynomial is $Q$.  Let $N_\E$ be the integer and $R_\E$ be the homogeneous polynomial given by the Proposition \ref{existencedunM}.

Let $(u,v)$ and $(u',v')$ be pairs of integers such that $(u,v)\equiv(u',v')\mod N_\E$ and that $\forall R_i$ primitive factor of $R$ one has $R_i(u,v)>0$. For these pairs, we denote $\alpha:=Q(u,v)$ and $\beta:=Q(u',v')$.

Suppose moreover that we have 
\begin{enumerate}[(1)]
\item $\alpha=c^2l$, where $l$ is squarefree and $\gcd(c,l)=1$,
\item $\beta=c^2\eta$, where $\eta$ is squarefree, $\gcd(c,\eta)=1$,
\end{enumerate}

Then $h_{Q}(u,v)=h_{Q}(u',v')$.
\end{lm}

\begin{proof} We have
\[h_{Q}(u,v)=\prod \limits_{\underset{p^2\vert Q(u,v)}{p\nmid\delta}} \left(-\left(\frac{-\overline{c_6}(u,v)_{(p)}}{p}\right)\right)^{\nu_p(Q(u,v))-1}
=\prod \limits_{\underset{p^2\vert \alpha}{p\nmid\delta}} \left(-\left(\frac{-\overline{c_6}(u,v)_{(p)}}{p}\right)\right)^{\nu_p(\alpha)-1}\]
\[=\prod \limits_{\underset{p \mid c}{p\nmid\delta}} \left(-\left(\frac{-\overline{c_6}(u,v)_{(p)}}{p}\right)\right)^{2\nu_p(c)-1}=\prod \limits_{\underset{p^2\vert \beta}{p\nmid\delta}} \left(-\left(\frac{-\overline{c_6}(u,v)_{(p)}}{p}\right)\right)^{\nu_p(\beta)-1}\]
As $(u,v)\equiv(u',v')\mod N_\E$ and the two pairs are in the same connected component in $\R^2-\{(u,v) \ \vert \ R(u,v)=0\}$, then (by consequence of Proposition \ref{existencedunM}) for any $p\mid \alpha\beta$ and $p\nmid\delta$, we have $\overline{c_6}(u,v)\equiv \overline{c_6}(u',v')\mod p$. This implies that
\[h_{Q}(u,v)=\prod \limits_{\underset{p^2\vert Q(u',v')}{p\nmid\delta}} \left(-\left(\frac{-\overline{c_6}(u',v')_{(p)}}{p}\right)\right)^{\nu_p(Q(u',v'))-1}\]
The right hand side of this equality is exactly $h_{Q}(u',v')$.
\end{proof}

Now, we present a general result on values of polynomials which will allow us (when $\E$ has type $I_m^*$ at $Q$) to give a criterion on pairs $(u,v),(u',v')\in\Z\times\Z$ to be such that the function $h_{Q}$ take opposite values at those pairs.

\begin{lm} \label{variationpotmult} \cite[Lemma 2.3]{Manduchi} Let $Q(T)$ and $P(T)\in\Z[T]$ be such that $Q(T)$ is non-constant. Let $\mathrm{Res}(P,Q)$ be the resultant of $P$ and of $Q$, and let $\Delta_Q$, be the discriminant of $Q$. Suppose $\mathrm{Res}(P,Q)$ and $\Delta_Q$ are non-zero. Let $\mathscr{P}_0$ be a finite set of prime numbers.

Then there exists a prime number $p_0\not\in\mathscr{P}_0$ and $n$ a positive integer such that $p_0^2\mid Q(n)$ and $p_0^{-2}P(n)Q(n)\equiv1\mod p_0$.
In particular, $p_0^2\mid\mid Q(n)$ and $p_0\nmid P(n)$.
\end{lm}

We refer to \cite{Manduchi} for a proof of this elementary lemma.
Observe that there is no need of the Squarefree conjecture in Lemma \ref{variationpotmult}.

\begin{lm}\label{variationhpotmult}
Let $\E$ be an elliptic surface with a place of type $I_m^*$  
whose associated polynomial is $Q$. Let $N_\E$ be the integer and $R_\E(U,V)$ be the homogeneous polynomial given by Proposition \ref{existencedunM}.
Put $P=-\frac{\overline{c_6}(u,v)}{Q(u,v)^3}$.

Let $(u,v)$ and $(u',v')$ be pairs of coprime integers such that $(u,v)\equiv(u',v')\mod N_\E$ and that $R_i(u,v)>0$ for every primitive factor $R_i$ of $R_\E$. We denote $\alpha:=Q(u,v)$ and $\beta:=Q(u',v')$.

Suppose there exists a prime number $q_0$ such that we have 
\begin{enumerate}[(1)]
\item $\alpha=c^2l$, where $l$ is without squarefactor dividing $N_\E$ and $\gcd(c,l,q_0)=1$,
\item $\beta=c^2q_0^2\eta$, where $\eta$ is without squarefactor dividing $N_\E$, $\gcd(c,\eta)=\gcd(q_0,c\eta)=1$,
\item $q_0\nmid\delta$ and $q_0^{-2}P(u,v)Q(u,v)\equiv q_0^{-2}P(u',v')Q(u',v')\equiv1\mod q_0$.
\end{enumerate}

Then $h_{Q}(u,v)=-h_{Q}(u',v')$.
\end{lm}

\begin{proof}
Let $Q$ be a homogeneous polynomial associated to a place of type $I_m^*$ of $\E$. We put $Q(u,v)=:\alpha=c^2l$ and $Q(u',v')=:\beta=c^2q_0^2\eta$. Then by Theorem \ref{formuledusigne}, one has

\[h_{Q}(u',v')=\prod_{p\nmid\delta;p\mid cq_0}{\begin{cases}
-\Big(\frac{(-\overline{c_6}(u',v'))_{(p)}}{p}\Big) & \text{if }2\nu_p(q_0c)\equiv2,4\mod6\\
+1 & \text{otherwise.}
\end{cases}}\]
\[=-\Big(\frac{(-\overline{c_6}(u,v))_{(q_0)}}{q_0}\Big)\prod_{p\nmid\delta;p\mid c}{\begin{cases}
-\Big(\frac{(-\overline{c_6}(u,v))_{(p)}}{p}\Big) & \text{if }\nu_p(c)\text{ is even,}\\
+1 & \text{otherwise.}
\end{cases}}\]
\[=-\Big(\frac{(-\overline{c_6}(u,v))_{(q_0)}}{q_0}\Big)\cdot h_{Q}(u,v).\]

By assumption on $q_0$, one has $q_0^{-2}\frac{-\overline{c_6}(u,v)}{Q(u,v)^2}\equiv1\mod q_0$. If we put $Q(u,v)=q_0^2\mu$ where $\mu$ is an integer coprime to $q_0$, one has $q_0^{-6}(-\overline{c_6}(u,v))\equiv \mu^2\mod q_0$. Thus, one has $\left(\frac{(-\overline{c_6}(u,v))_{(q_0)}}{q_0}\right)=+1$ for all $(u,v)$.

Thus we have the equality
\[h_{Q}(u,v)=-h_{Q}(u',v').\]
\end{proof}

\subsection{Variation of the global root number}

\subsubsection{The general case}

\begin{propo}\label{VariationIm}
Let $\E$ be an elliptic surface satisfying the hypotheses of Theorem \ref{thmhelf} or of Theorem \ref{nouveauthm}. 
Let $N_\E$ be the integer and $R_\E(U,V)$ be the homogeneous polynomial given by Proposition \ref{existencedunM}. 
Let $t_1,t_2\in\Q$ be integers, and $(u_1,v_1),(u_2,v_2)\in\Z\times\Z_{>0}$ be their corresponding numerators and denominators as defined in Section \ref{t=xy}. Suppose they satisfy the following properties.

\begin{enumerate}
\item
We have that $(u_1,v_1)\equiv (u_2,v_2)\mod N_\E$, 
\item For every primitive factor $R_i$ of $R_\E$ one has $R_i(u,v)>0$,

\item 
The values $M_\E(u_1,v_1)$ and $M_\E(u_2,v_2)$ are squarefree integers.
\item\label{cellela}
For every $Q$ of type $I_m^*$ or $I_m$, there exists an integer $c_Q$ such that we have \begin{enumerate}
\item $Q(u_1,v_1)=c_Q^2l$ and \item $Q(u_2,v_2)=c_Q^2l'$ 
\end{enumerate}
where $l$ and $l'$ are squarefree integers coprime to $N_\E$. 
\end{enumerate}
If $\E$ admits a place of type $II$, $II^*$,  $IV$ and $IV^*$ at a $P_i$ such that $\mu_3\not\subseteq\Q[T]/P_i(T_,1)$
or of type $III$ or $III^*$ at a $P_i$ such that $\mu_4\not\subseteq\Q[T]/P_i(T,1),$ then suppose hypothesis \ref{cellela} holds for $P_i$ as well.

Then
\[W\left(\E_{\frac{u_1}{v_1}}\right)=\lambda(M_\E(u_1,v_1)M_\E(u_2,v_2))W\left(\E_{\frac{u_2}{v_2}}\right).\]
\end{propo}

\begin{rem}
When the only place of type $I_m$ is the one at infinity, 
we rather take as a hypothesis $v_1\equiv v_2\mod N_\E$ squarefree integers, and the conclusion of Proposition \label{VariationIm} is
\[W\left(\E_{\frac{u_1}{v_1}}\right)=\lambda(v_1v_2)W\left(\E_{\frac{u_2}{v_2}}\right).\]
However, we can avoid this case by a change of variable.
\end{rem}

\begin{rem}
When there is no place of type $I_m$, the conclusion of this proposition is \[W\left(\E_{\frac{u_1}{v_1}}\right)=W\left(\E_{\frac{u_2}{v_2}}\right).\] Hence, in this case, we can not use this proposition to make the root number vary, and we will need Proposition \ref{variationglobale}.
\end{rem}

\begin{proof}
By Theorem \ref{formuledusigne}, we have:
\begin{equation}\label{4.1}
W(\E_{\frac{u_1}{v_1}})=\lambda(M_\E(u_1,v_1))\prod_{p\mid\delta}{W_p\left(\E_{\frac{u_1}{v_1}}\right)}\prod_{P\in\mathscr{M}}{g_{P}(u_1,v_1)}\prod_{P\in\mathscr{B}}{h_{P}(u_1,v_1)}.\end{equation}
The rest of the proof lies on the fact that every term in equation (\ref{4.1}) is constant.
We know by construction of the integer $N_\E$ and  of the polynomial $R_\E$ that \[\prod_{p\mid\delta}{W_p\left(\E_{\frac{u_1}{v_1}}\right)}\prod_{P\in\mathscr{M}}{g_{P}(u_1,v_1)}=\prod_{p\mid\delta}{W_p\left(\E_{\frac{u_2}{v_2}}\right)}\prod_{P\in\mathscr{M}}{g_{P}(u_2,v_2)}.\]

By assumption, a polynomial $Q$ at which $\E$ has type $I_m^*$ is such that $Q(u_1,v_1)$ and $Q(u_2,v_2)$ have the same square part: a constant named $c_Q$.
By this fact and Lemmas \ref{constance24} and \ref{constance3}, we have the equality
\[\prod_{P\in\mathscr{B}}{h_{P}(u_1,v_1)}=\prod_{P\in\mathscr{B}}{h_{P}(u_2,v_2)}.\]

To finish, we observe that \[h_{\infty}(u_1,v_1)=h_{\infty}(u_2,v_2)=+1.\]
\end{proof}

\subsubsection{No place of reduction of type $I_m$}

\begin{propo}\label{variationglobale}
Let $\E$ be an elliptic surface with no place $I_m$ and satisfying the hypothesis of Theorem \ref{thmhelf} or \ref{nouveauthm}. Let $N_\E$ be the integer and $R_\E(U,V)$ be the homogeneous polynomial given by Proposition \ref{existencedunM}. 
Let $t_1,t_2\in\Q$ be integers, and $(u_1,v_1),(u_2,v_2)$ be their corresponding numerators and denominators as defined in Section \ref{t=xy}. Suppose they satisfy the following properties. 

\begin{enumerate}
\item We have $(u_1,v_1)\equiv(u_2,v_2)\mod N_\E$, a non-zero congruence class.
\item For every primitive factor $R_i$ of $R_\E$ one has $R_i(u,v)>0$. 

\item For a certain $Q_0$ of type $I_m^*$ ($m>0$), one has
\begin{enumerate}
\item $Q_0(u_1,v_1)=c^2l$ where $l$ is a squarefree integer coprime to $N_\E$,
\item $Q_0(u_2,v_2)=c^2q_0^2l'$ where $l'$ is a squarefree integer coprime to $N_\E$, and $q_0$ is a prime number which does not divide $\delta$ and such that $-p_0^{-6}\overline{c_6}(u_i,v_i)$ is a square modulo $q_0$ for $i=1,2$.
\end{enumerate}
\item For every $Q\not=Q_0$ of type $I_m^*$ ($m>0$),
\begin{enumerate}
\item $Q(u_1,v_1)=c_Q^2l_Q$ where $l_Q$ is a squarefree integer coprime to $N_\E$,
\item $Q(u_2,v_2)=c_Q^2l_Q'$ where $l_Q'$ is a squarefree integer coprime to $N_\E$.
\end{enumerate}
\end{enumerate}

If $\E$ admits a place of type $II$, $II^*$,  $IV$ and $IV^*$ at a $P_i$ such that $\mu_3\not\subseteq\Q[T]/P_i(T_,1)$
or of type $III$ or $III^*$ at a $P_i$ such that $\mu_4\not\subseteq\Q[T]/P_i(T,1),$ then suppose hypothesis \ref{cellela} holds for $P_i$ as well.

Then, we have \[W(\E_{t_1})=-W(\E_{t_2}).\]
\end{propo}

\begin{proof} The proof is similar to Proposition \ref{VariationIm}, except that for the function $h_{Q_0}$ corresponding to the chosen $Q_0$ of potential multiplicative reduction, Lemma \ref{variationhpotmult} shows that \[h_{Q_0}(u_1,v_1)=-h_{Q_0}(u_2,v_2).\]
\end{proof}

\section{Proof of the theorem \ref{thmhelf} and \ref{nouveauthm}}\label{sectionproofHelfgott}

The proofs of Theorem \ref{thmhelf} and \ref{nouveauthm} are similar.
In this Section we prove the two theorems at the same time, highlighting the differences along the way. Depending on whether or not $\E$ has a singular fiber of multiplicative reduction, we use a different strategy. 

\subsection{The case $M_\E=1$}

\begin{thm}
Let $\E$ be a non-isotrivial elliptic surface with no place of multiplicative reduction and satisfying the hypotheses of Theorem \ref{thmhelf} or \ref{nouveauthm}.

Then the two sets $\#W_\pm(\E)$ are infinite.
\end{thm}

\begin{proof}
Let $\frac{n_4}{d_4}$, $\frac{n_6}{d_6}$, $\frac{n_\Delta}{d_\Delta}$ be the content of respectively the polynomials $c_4(T)$, $c_6(T)$, $\Delta(T)$ associated to $\E$. (Observe that the only factors of $d_4,d_6,d_\Delta$ are $2$ or $3$ when the fractions for the contents are irreducible.) Put
\[\delta=2\cdot3\cdot n_4\cdot n_6\cdot n_\delta\prod_{Q,Q'\in\mathscr{B}}{\mathrm{Res}(Q,Q'}).\]

Let $N=N_\E$ be the integer and $R_\E$ be the homogeous polynomial given by Proposition \ref{existencedunM} (choose the minimal such integer $N_\E$ and polynomial $R_\E$ of lowest degree). Write $R=\gamma R_1\cdot R_r$, the decomposition in primitive polynomial. Let us choose $R$ such that $sign(\gamma)=+1$.

The function $t\mapsto\prod_{p\mid\delta}{W_p(\E_{\frac{u}{v}})}\prod_{P\in\mathscr{B}}{g_{P}(u,v)}$ is constant when $(u,v)\in\Z\times\Z_{>0}$ stays in a double congruence class modulo $N$ and when $(u,v)$ are in a connected component of \[\R^2-\{(u,v)\in\R^2 \ \vert R(u,v)=0\}.\] Recall that here, $u,v$ are the coprime integers such that $t=\frac{u}{v}$ as in \ref{t=xy}.

For each $p\mid N$, put $\alpha_p=\nu_p(N)$. Moreover, write $N=2^{\alpha_2}3^{\alpha_3}p_1^{\alpha_{p_1}}\cdots p_s^{\alpha_{p_s}}$ the factorisation into distinct prime numbers.

Put \[S=(2,3,p_1,\dots,p_s),\] 
 and \[T=(0,\dots,0).\]

Let $\mathfrak{a}_2,\mathfrak{b}_2\mod 2^{\alpha_2}$ be congruence classes such that
for all $P_i$ of bad reduction (except $I_0^*$)
\begin{equation*}\label{classe2}
P_i(\mathfrak{a}_2,\mathfrak{b}_2)\not\equiv0\mod2^{\alpha_2}.
\end{equation*}

Let $(\mathfrak{a}_3,\mathfrak{b}_3)\mod 3^{\alpha_3}$ be classes such that for all $P_i$ of bad reduction (except $I_0^*$)
 \begin{equation*}\label{classe3}
P_i(\mathfrak{a}_3,\mathfrak{b}_3)\not\equiv0\mod3^{\alpha_3}.
\end{equation*}

Let also, for each $p\mid N$ such that $p\not=2,3$, be classes $\mathfrak{a}_p,\mathfrak{b}_p\mod p^{\alpha_p}$ such that for all $P$ of bad reduction (except $I_0^*$) we have:
\[P_i(\mathfrak{a}_p,\mathfrak{b}_p)\not\equiv0\mod p^{\alpha_p}.\]
As by assumption $P_i$ has content 1, such classes $\mathfrak{a}_p, \mathfrak{b}_p$ exist for every $p\mid N$.

By the Chinese Remainder Theorem, there exists integers $a,b$ satisfying
\begin{equation}\label{chinois1}
(a,b)\equiv\begin{cases}
(\mathfrak{a}_2,\mathfrak{b}_2)\mod 2^{\alpha_2}, & \\
(\mathfrak{a}_3,\mathfrak{b}_3)\mod 3^{\alpha_3}, & \\
(\mathfrak{a}_p,\mathfrak{b}_p)\mod p^{\alpha_p} & \text{for all $p\mid N$,} \\
\end{cases}
\end{equation}

We can choose them such that $R_i(u,v)>0$ for all $1\leq i\leq r$.

Put $Q_0=\prod{Q}$ where $Q$ runs through the polynomials of potentially good reduction. (If $\E$ does not satisfy the hypotheses of Theorem \ref{nouveauthm}, take instead this product over every polynomial of bad reduction.)

By the Squarefree Sieve given by Corollary \ref{crible} applied to $Q_0,S,T,N,a$ and $b$ as previously, there exists a set $\mathscr{F}_1$ of infinitely many pairs $(u,v)\in\Z\times\Z_{>0}$ such that 
\[Q_0(u,v)=l,\]
where $l$ is a squarefree integer coprime to each $p\in S$ by our choice of $S$ and $T$.
By Remark \ref{thmsurlessecteurs}, this set $\mathscr{F}_1$ may as well be chosen such that  $\forall(u,v)\in\mathscr{F}_1$ may as well be chosen such that $R_i(u,v)>0$ for all $1\leq i\leq r$.

By Proposition \ref{variationglobale}, for all $(u,v),(u',v')\in\mathscr{F}_1$,
\[W(\E_{\frac{u}{v}})=W(\E_{\frac{u'}{v'}}).\]

Choose $Q_1$ a polynomial associated to a place of type $I_m^*$. We are assured that there exists at least one such polynomial: it is a pole of the $j$-invariant. It is possible to proceed to a linear change of variable to avoid the case where the only place of type $I_m^*$ is the one at infinity. We thus suppose without loss of generality $Q_1\in\Z[U,V]$.

Put $P(U,V)=-\overline{c_6}(U,V)/Q(U,V)^3$.
By Lemma \ref{variationpotmult} applied to $P(T,1)=:P(T)$, $Q(T,1)=:Q(T)$ and $S=:\mathscr{P}_0$, there exist $q_0\not\in S$ and $m_0\leq0$ an integer such that $q_0^2\mid Q(m_0,1)$ and that $-q_0^{-2}P(m_0,1)Q(m_0,1)$ is a square modulo $q_0$. 

Consider the sequences
\[S'=(2,3,p_1,\dots, p_s,q_0)\]
and
\[T'=(0,\dots,0,2).\]

By the Chinese Remainder Theorem, there exists a pair of integers $(a',b')$ satisfying both (\ref{chinois1}) and 
\[a'\equiv m_0\mod q_0^3\text{, }b'\equiv 1\mod q_0^3.\]
We can choose them such that $R_i(u,v)>0$ for all $1\leq i\leq r$.

By using Corollary \ref{crible} 
on \[\text{$Q_1$, $S'$, $T'$, $a'$ and $b'$,}\] we obtain $\mathscr{F}_2$ a set of infinitely many pairs of coprime integers $(u,v)\in\Z\times\Z_{>0}$ such that \[Q_1(u,v)=q_0^2l_{u,v},\]
where $l$ is an squarefree integer coprime to every element of $S'$ and where $q_0^{-6}\overline{c_6}(u,v)$ is a square modulo $q_0$. As previously, we choose $\mathscr{F}_2$ such that $\forall(u,v)\in\mathscr{F}_2$ one has $R_i(u,v)>0$ for all $1\leq i\leq r$. By Proposition \ref{variationglobale}, all the elements of $\mathscr{F}_2$ are such that their fibers on $\E$ have the same root number.\\

To end the proof, we use Proposition \ref{variationglobale} to show that for all $(u_1,v_1)\in\mathscr{F}_1$, and every $(u_2,v_2)\in\mathscr{F}_2$, one has \[W\left(\E_{\frac{u_1}{v_1}}\right)=-W\left(\E_{\frac{u_2}{v_2}}\right).\]
\end{proof}

\subsection{The case $M_\E\not=1$}

\begin{thm}\label{nouveauthmconjecturel}
Let $\E$ be an non-isotrivial elliptic surface satisfying the hypotheses of Theorem \ref{thmhelf} or Theorem \ref{nouveauthm}. Suppose that $M_\E\not=1$, in other words there exists a place of multiplicative reduction on $\E$. 
Then the sets $W_\pm(\E)$ are both infinite. 
\end{thm}

\begin{proof}  

We study a surface such that the infinite place in not $I_m$ (we can make this assumption without lost of generality, by changing variable if needed).

Let $N:=N_\E$ be the integer and let $R_\E$ be the homogenous polynomial corresponding to $\E$ given by Proposition \ref{existencedunM}. Write $R_\E=R_1\cdots R_r$ the factorisation into irreducible factors. For each $p\mid N$, let $\alpha_p=\nu_p(N)$, so that \[N=2^{\alpha_2}3^{\alpha_3}\cdot p_1^{\alpha_{p_1}}\dots p_s^{\alpha_{p_s}}\]
is the factorisation into distinct prime factors.
Put \[S=(2,3,p_1,\dots,p_s)\] 
and \[T=(0,\dots,0).\]

In a similar way as in the previous theorem, the Chinese Remainder Theorem allows to obtain $(a,b)$ a congruence class modulo $N$ such that 
\begin{equation}\label{chinois}
\begin{cases}
B_\E(a,b)\not\equiv0\mod 2^{\alpha_2}, & \\
B_\E(a,b)\not\equiv0\mod 3^{\alpha_3}, & \\
B_\E(a,b)\not\equiv0\mod p^{\alpha_p} & \text{for all $p\mid N$,} \\
\end{cases}
\end{equation}

Moreover we can choose $(a,b)$ such that $R_i(a,b)>0$ for all $0\leq i\leq r$.

By Chowla's conjecture for the polynomial $B_\E$, it is possible to find two pairs of integers $(a_1,b_1)$ and $(a_2,b_2)$ such that $(a_1,b_1)\equiv(a_2,b_2)\equiv(a,b)\mod N$ such that respectively $\lambda(B_\E(a_1,b_1))=-1$ and $\lambda(B_\E(a_2,b_2))=+1$.

Using twice the Squarefree-Liouville sieve of Corollary \ref{sflcrible} on $M_\E=:f$, $B_\E=:g$, $R_\E:=R$, $S$, $T$, $N$, $a$ and $b$, we find a set $\mathscr{F}_1$ and $\mathscr{F}_2$ of infinitely many pairs $(u,v)\in\Z\times\Z_{>0}$ such that 
\[(u,v)\equiv (a,b)\mod N\]
and \[R_i(u,v)>0\quad \text{for all }1\leq i\leq r\]
and
 \[\lambda(k_{u,v})=+1\text{, if $(u,v)\in\mathscr{F}_1$}\qquad\text{ or }\lambda(k_{u,v})=-1\text{, if $(u,v)\in\mathscr{F}_2$.}\]
and moreover such that
\[B_\E(u,v)=l_{u,v},\qquad\text{ and }M_\E(u,v)=k_{u,v},\]
where $l_{u,v},k_{u,v}$ are squarefree integers (depending on $u,v$) coprime to every $p\in S$ by the choice of $S$ and of $T$.

Thus by Proposition \ref{VariationIm}, for all $(u_1,v_1)\in\mathscr{F}_1$, $(u_2,v_2)\in\mathscr{F}_2$, one has \[W\left(\E_{\frac{u_1}{v_1}}\right)=-W\left(\E_{\frac{u_2}{v_2}}\right).\]
\end{proof}

\section{Families with arbitrarily large degree discriminant factors}\label{sectionexample}

\begin{thm}
Let $Q$ be a squarefree polynomial such that its irreducible factors have degree less or equal to 6 and not equal to $T$. Let $N\geq1$. Put \[P(T)=3\alpha^2Q(T)^2+\beta^2T^{2N},\]
and $\alpha,\beta\in\Z$ coprime.

Let $\E$ be the elliptic surface given by the equation \[\E:y^2=x^3-27P(T)Q(T)^2x-54\beta P(T)Q(T)^3T^N.\]

Then $W_+$ and $W_-$ are infinite.

Moreover, if we assume the parity conjecture to hold, then the rational points of $\E$ are Zariski-dense.
\end{thm}

In this example, one has $\deg P=2\max (\deg Q,N)$, that can be as large as we want.

\begin{proof}
For $t\in\Q$, we write $t=\frac{u}{v}$ for $(u,v)\in\Z\times\Z_{>0}$ a pair of coprime integers (as defined in \ref{t=xy}). As in Section \ref{notaroot},  we denote by $\E_{u,v}$ the elliptic curve isomorphic to the fiber $\E_t$ :
\[\E_{u,v}:y^2=x^3-27v^{4k-\deg P-2\deg Q} P(u,v)Q(u,v)^2-54\beta v^{6k-\deg P -3\deg Q-N}P(u,v)Q(u,v)^3u^{N},\]
where $k$ is the smallest integer such that $4k\geq\deg c_4(T)$ and $6k\geq\deg c_6(T)$.

This elliptic curve has discriminant \[\Delta(\E_{u,v})=\gamma v^{12k-2\deg P-8\deg Q}P(u,v)^2Q(u,v)^8,\] where $\gamma=1-\beta^2\in\Z$ is a constant. We want to know the value of $k$, and of $12k-2\deg P-8\deg Q$, as this gives the reduction type of the infinite place.

Suppose that $N\leq\deg Q$. One has $\deg P=2\deg Q$ and thus
\[\deg c_4(T)=4\deg Q,\]
\[\deg c_6(T)=5\deg Q+N,\]
\[\deg \Delta(T)=12\deg Q.\]
We have $k=\deg Q$ and $12k-2\deg P-8\deg Q=0$. Hence, the infinite place has good reduction. The bad places of the surface $\E$ are the following: 
\begin{enumerate}
\item the places associated to $P_i$, the irreducible factors of $P(T)=P_1^{e_1}\dots P_n^{e_n}$, of type $II$, $IV$, $I_0^*$, $IV^*$, $II^*$ or $I_0$ according to $e_i\mod6$;
\item the places associated to the factors of $Q(T)$ of type $I_2^*$.
\end{enumerate}

Suppose now that $N\geq\deg Q$ and put $N=\deg Q+a$ for some $a\in\mathbb{N}$. Observe that $\deg P=2N$. One has
\[\deg c_4(T)=4\deg Q+2a,\]
\[\deg c_6(T)=6\deg Q+3a,\]
\[\deg \Delta(T)=12\deg Q+4a.\]

The surface $\E$ has the following places of bad reduction: 
\begin{enumerate}
\item the places associated to $P_i$, the irreducible factors of $P(T)=P_1^{e_1}\dots P_n^{e_n}$ of type $II$, $IV$, $I_0^*$, $IV^*$, $II^*$ or $I_0$ according to $e_i\mod6$;
\item the places associated to the factors of $Q(T)$ of type $I_2^*$; and
\item the infinite place is $I_{2a}^*$ or $I_{2a}$ depending on the parity of $a$.
\end{enumerate}

First, observe that for every $p\mid P(u,v)$, we have
\[3\alpha^2Q(u,v)^2+\beta^2u^{\deg Q}\equiv0\mod p\]
\[\Rightarrow \left(\frac{\beta u^{\deg Q}}{\alpha Q(u,v)}\right)^2\equiv-3\mod p.\]
This means that for all $p\mid P(u,v)$, we have $(\frac{-3}{p})=+1$. 
Let $P_i$ be an irreducible factor of $P$. We have $3\alpha^2Q(T)^2+\beta^2T^{2n}=P_i(T)P_{(P_i)}(T)$, for a certain polynomial $P_1$ such that $P_i(T)P_{(P_i)}(T)=P(T)$. The field $\Q[T]/P_i(T)$ is generated by $\xi$, a root of $P_i$. Moreover, $\xi$ is such that $3\alpha^2Q(\xi)^2+\beta^2\xi^{2N}=0$. Therefore we have $-3=(\frac{\beta\xi^{N}}{\alpha Q(\xi)})^2$ and $\Q(\mu_3)\subset \Q[T]/(P_i[T])$.

If $a$ is even, the value of $k$ is $\deg Q+\frac{a}{2}$, we have $12k-2\deg P-8\deg Q=2a$ and $\deg c_4\equiv0\mod 4$.
Therefore, the infinite place has type $I_{2a}$. As it is the only place of multiplicative reduction and that $P$ and $Q$ satisfy the assumptions of Theorem \ref{nouveauthm}, this shows that infinitely many fibers of $\E$ take the value $+1$ and infinitely many take the value $-1$.

If $a$ is odd, the value of $k$ is $\deg Q+\frac{a+1}{2}$, we have $12k-2\deg P-8\deg Q=2a$ and $\deg c_6\equiv2\mod4$. Therefore, the infinite place has type $I_{2a}^*$. In this case too, we obtain $\# W_{\pm}=\infty$ by using Theorem \ref{nouveauthm}.
\end{proof}

\begin{rem}
The hypotheses of Theorem \ref{nouveauthm} can be weakened.
\begin{enumerate}
\item We can lighten the hypothesis on $Q$. Rather than suppose it to be squarefree, we can allow $Q$ to take the form 
\[Q(T)=A(T)\prod_{j=1}^s(T-a_j)^{e_j},\]
where $a_j\in\Q$ and $e_j\in\N$ for all $j=1,\cdots,s$, and where $A(T)$ is a squarefree polynomial whose irreducible factors have degree $\leq6$.

In this case, the multiplicative places are
\begin{enumerate}
\item possibly the place at infinity,
\item the places associated to the polynomial $T-a_j$ such that the exponent $e_j$ is even.
\end{enumerate}

Therefore, the polynomial $M_\E$ is a product of linear factors and satisfies Chowla's conjecture.

\item We can replace $T^N$ by a polynomial $B(T)\in\Z[T]$ such that $\mathrm{Res}(B,Q)\not=0$. More precisely, we define \[P(T)=3\alpha^2Q(T)^2+\beta^2B(T)^2,\]
where $\alpha,\beta\in\Q$, and we consider the elliptic surface described by the Weierstrass equation \[y^2=x^3-27P(T)Q(T)^2x-54\beta P(T)Q(T)^3B(T).\]
\end{enumerate}
\end{rem}

\appendix

\section{Properties of the modified symbol}\label{quadraticsymbol}

We have defined in Section \ref{sectionformuledusigne} and used in Section \ref{sectionvariation} and \ref{sectionproofHelfgott} the following modified quadratic symbol.

Let $\delta$ be an even 
integer and $(a,b)\in\Z^{*}\times\Z^{*}$ be a pair of integers. We define $\left(\frac{a}{b}\right)_\delta\in\{-1,+1\}$ as
\[\left(\frac{a}{b}\right)_\delta=\prod_{p\nmid\delta}{\left(\frac{a_{(p)}}{p}\right)^{\nu_p(b)}}\]
where the product runs through the prime number $p\nmid\delta$, $a_{(p)}$ is the integer such that $a=a_{(p)}p^{\nu_p(a)}$ and $(\frac{\cdot}{p})$ is Legendre symbol. The newly defined operator is actually not so far from being simply the usual Jacobi symbol. Indeed, if $\gcd(a,b\delta)=1$, then 
\[\left(\frac{a}{b}\right)_\delta=\left(\frac{a}{\vert b_{(\delta)}\vert}\right),\]
where we recall that \[b_{(\delta)}=\frac{b}{\prod_{p_i\mid\delta}{p_i^{\nu_{p_i}(b)}}}.\]

This symbol respects many convenient properties:

\begin{propo}\label{proprietes}
Let $a,b,c\in\Z^{*}$ and $\delta$ be an even integer. One has the following:
\begin{enumerate}[(a)]
\item\label{propa} $\left(\frac{ab}{c}\right)_\delta=\left(\frac{a}{c}\right)_\delta\cdot\left(\frac{b}{c}\right)_\delta$.

\item\label{propb} $\left(\frac{a}{bc}\right)_\delta=\left(\frac{a}{b}\right)_\delta\cdot\left(\frac{a}{c}\right)_\delta$.

\item\label{propc} $\left(\frac{a}{b}\right)_\delta=\left(\frac{a+bc}{b}\right)_\delta$.

\item\label{propd} For $b$ fixed, $a\mapsto\left(\frac{a}{b}\right)_\delta$ can be expressed as a finite product of Legendre symbols: \[\prod \limits_{\underset{ \nu_p(b)\text{ odd}}{p\nmid\delta}} \left(\frac{a_{(p)}}{p}\right).\]

\item\label{prope} If $\delta_1,\delta_2$ are two even integers such that $\delta_1\mid\delta_2$, one has that $\frac{\left(\frac{a}{b}\right)_{\delta_1}}{\left(\frac{a}{b}\right)_{\delta_2}}$ can be expressed as a finite product of Legendre symbols:
 
\[ \frac{\left(\frac{a}{b}\right)_{\delta_1}}{\left(\frac{a}{b}\right)_{\delta_2}}=\prod_{ p\nmid\delta_1\atop p\mid \delta_2}{ \Big(\frac{a_{(p)}}{p}\Big)^{v_p(b)}}.\]

\item\label{propf} One has that $\frac{\left(\frac{a}{b}\right)_\delta}{\left(\frac{b}{a}\right)_\delta}$ can be expressed as a finite product of Hilbert quadratic symbols and Legendre symbols:

\[\frac{\left(\frac{a}{b}\right)_\delta}{\left(\frac{b}{a}\right)_\delta}=\prod_{p\mid \delta\text{ or }p=\infty}{\Big(\frac{a,b}{p}\Big)}\prod_{p\nmid\delta\atop p\mid a}{\left(\frac{-1}{p}\right)^{v_p(a)v_p(b)}}.\]
\end{enumerate}
\end{propo} 

\begin{rem}
In this paper, we will always consider this symbol on integers $a,b$ such that for any factor $p\mid\gcd(a,b)$, we have also $p\mid\delta$. In such case, Property (\ref{propf}) becomes: \[\left(\frac{a}{b}\right)_\delta\left(\frac{b}{a}\right)_\delta=\prod_{p\mid \delta \text{ or }p=\infty}{\Big(\frac{a,b}{p}\Big)}.\]

In particular, noticing that $\left(\frac{a}{b}\right)_\delta=\left(\frac{a}{\vert b\vert}\right)_\delta$, one has
 \[\left(\frac{a}{b}\right)_\delta\left(\frac{\vert b\vert}{a}\right)_\delta=\prod_{p\mid \delta}{\Big(\frac{a,\vert b\vert}{p}\Big)}.\]
Indeed, $\left(\frac{a,b}{\infty}\right)=\begin{cases}+1&\text{ if at least one of $a,b$ is $>0$}\\
-1&\text{if both $a,b$ are $<0$}\end{cases}$
and thus $\left(\frac{a,\vert b\vert}{\infty}\right)=+1$.
\end{rem}

\begin{proof}
The proof of the first five statements of Proposition \ref{proprietes} is basically just a verification. Concerning the last point, remember 
that for any prime number $p\not=2$, \cite[Thm. 1, Chap. III]{Serre}says that the Hilbert symbol at $p$ is equal to 
\[\Big(\frac{a,b}{p}\Big)=(-1)^{\nu_p(a)\nu_p(b)\left(\frac{p-1}{2}\right)}\left(\frac{a_{(p)}}{p}\right)^{\nu_p(b)}\left(\frac{b_{(p)}}{p}\right)^{\nu_p(a)}.\] 

A Hilbert symbol depends only of the congruence classes mod $p$ of $a_{(p)}$ and $b_{(p)}$. For more information, we refer to the introduction to Hilbert symbols given in \cite[Chap. III]{Serre}.


One has
\begin{align*}
\frac{\left(\frac{a}{b}\right)_\delta}{\left(\frac{b}{a}\right)_\delta}=&\prod_{p\nmid\delta}{\left(\frac{a_{(p)}}{p}\right)^{\nu_p(b)}}\prod_{p\nmid\delta}{\left(\frac{b_{(p)}}{p}\right)^{\nu_p(a)}}\\
=&\prod_{p\nmid\delta}{\left(\frac{a,b}{p}\right)(-1)^{\nu_p(a)\nu_p(b)(\frac{p-1}{2})}}\\
=&\prod_{p\nmid\delta}{\left(\frac{a,b}{p}\right)}\prod_{p\nmid\delta;p\mid a}{(-1)^{\nu_p(a)\nu_p(b)(\frac{p-1}{2})}}\\
=&\prod_{p\nmid\delta}{\left(\frac{a,b}{p}\right)}\prod_{p\nmid\delta;p\mid a}{\left(\frac{-1}{p}\right)^{\nu_p(a)\nu_p(b)}}\\
=&\prod_{p\mid\delta\text{ or }p=\infty}{\left(\frac{a,b}{p}\right)}\prod_{p\nmid\delta;p\mid a}{\left(\frac{-1}{p}\right)^{\nu_p(a)\nu_p(b)}}
\end{align*}
The last equality is given by the product formula:
\[\prod_{p\leq\infty}{\left(\frac{a,b}{p}\right)}=1.\]
\end{proof}



In Section \ref{sectionformuledusigne}, we use this symbol with $a$ and $b$ values of polynomials. Note that we always choose $\delta$ such that any $p\mid \mathrm{Res}(A,B)$ also divides $\delta$. In this case, $A(u,v)_{(p)}=A(u,v)$ for any $p\mid B(u,v)$ and thus for every pair of integers $(u,v)\in\Z^2$ we have
\[\left(\frac{A(u,v)}{B(u,v)}\right)_\delta=\left(\frac{A(u,v)}{\vert B(u,v)_{(\delta)}\vert}\right).\]

Recall Definition \ref{localconstancy}.
Let $\mathscr{S}\subseteq\Z\times\Z$. A function $f:\mathscr{S}\rightarrow\{\pm1\}$ is said \emph{locally constant} is there exists $N\in\mathbb{N}^*$ such that the congruences $(u,v)\equiv(u',v')\mod N$ implies
\[f\left(u,v\right)=f\left(u',v'\right).\] 

Similarly, let $\mathscr{T}\subseteq\Q$. A function $\varphi:\mathscr{T}\rightarrow\{\pm1\}$ is \emph{locally constant} if there exists $N\in\mathbb{N}^*$ such that the congruences $(u,v)\equiv(u',v')\mod N$ implies
\[\varphi\left(\frac{u}{v}\right)=\varphi\left(\frac{u'}{v'}\right).\] 

With this definition in mind, we deduce the following Lemma from Proposition \ref{proprietes}:

\begin{lm}\label{firststep} Let $P\in\Z[U,V]$ be a non-zero homogeneous polynomial and $a\in\Z-\{0\}$.
The symbol $\left(\frac{P(\cdot,\star)}{a}\right)_\delta$ induces a locally constant function from $\Z^2-\{(u,v) \ \vert \ P(u,v)=0)\}$ to $\{-1,+1\}$.
\end{lm}

Also recall the following basic fact about Hilbert symbols:

\begin{lm}\label{hilbertsymbol}
Let $P,Q\in\Z[U,V]$ be non-zero homogeneous polynomial and $p$ a prime number.
The symbol $\left(\frac{P(\cdot,\star),Q(\cdot,\star)}{p}\right)$ induces a locally constant function from $\Z^2-\{(u,v)\ \vert \ P(u,v)Q(u,v)\not=0\}$ to $\{-1,+1\}$.

The symbol at infinity $\left(\frac{P(\cdot,\star),Q(\cdot,\star)}{\infty}\right)$ induces a function from $\Z^2-\{(u,v)\ \vert \ P(u,v)Q(u,v)=0\}$
 to $\{-1,+1\}$ that is constant on the connected components of $\mathbb{R}^2-\{(u,v)\ \vert \ P(u,v)Q(u,v)=0\}$.

\end{lm}

\begin{propo}\label{adelicsymbol}
Let $A,B\in\Z(U,V)$ be two non-zero homogeneous polynomials. Assume that $A$ and $B$ are coprime, that $B$ is primitive and that $\deg A$ is even. Let $\delta$ be a non-zero integer divisible by all primes dividing $\mathrm{Res}(A,B)\cdot\mathrm{content}(A)$.

Then there exists a homogeneous polynomial $R$ (depending on $A$ and $B$) such that, if we put \[D_{A,B}=\{(u,v)\in\Z\times\Z_{>0}\ \vert \ A(u,v)B(u,v)R(u,v)\not=0\},\]  the function $f:D_{A,B}\rightarrow \{\pm1\}$ defined by $f(u,v)=\left(\frac{A(u,v)}{B(u,v)}\right)_\delta$ is locally constant on the connected components of $\mathbb{R}^2-\{(u,v)\in\mathbb{R}^2\ \vert \ A(u,v)B(u,v)R(u,v)=0\}$. 
\end{propo}

\begin{rem}
In \cite{Helfgott}, we find a similar statement given by Corollary 5.3: if $B$ is a primitive homogeneous polynomial and $A$ is a rational function such that its valuation at $B$ is even, then the function $f$ (defined as previously) is locally constant on the connected components.

However, Helfgott's Corollary is inaccurate. A simple counterexample is $A(U,V)=U$, $B(U,V)=V$. 
In fact, our hypothesis "$\deg A$ even" is necessary.
\end{rem}



\begin{proof}
The proof 
goes by double induction on the degree of the polynomials $A$ and $B$. The goal is to prove that $\left(\frac{A(U,V)}{B(U,V)}\right)_\delta$ can be developed as a finite product (with no $U$- or $V$-dependence) of Jacobi or Hilbert symbols of the form
\begin{equation}\label{formesLH}\left(\frac{C_i(U,V)}{p}\right) \quad \text{ or }\left(\frac{D_j(U,V),E_j(U,V)}{p}\right),\end{equation}
where $C_i,D_j,E_j\in\Z[U,V]$ and $p$ is a prime number.


Let $(u,v)\in\Z\times\Z_{>0}$ be a pair of coprime integers.
By multiplicity of the symbol $(\frac{\cdot}{\cdot})_\delta$ we have
\begin{align*}
\left(\frac{A(u,v)}{B(u,v)}\right)_\delta=&\left(\frac{a\prod_{i=1}^{r}{P_i(u,v)^{e_i}}}{B(u,v)}\right)_\delta\\
=&\left(\frac{a}{B(u,v)}\right)_\delta\cdot\prod_{i\in[1,\cdots r]}\left(\frac{P_i(u,v)}{B(u,v)}\right)_\delta^{e_i}
\end{align*}

We use Lemma \ref{claim} (proven later on) which says that for any $P,Q\in\Z[U,V]$ distinct primitive homogeneous polynomials, there exists finitely many $R_j$ ($j=0,\dots,s\leq\deg P$) for which we have
\[\left(\frac{P(u,v)}{Q(u,v)}\right)_\delta=
h(u,v)\cdot \left(\left(\frac{u}{v}\right)_\delta\cdot\left(\frac{v,u}{\infty}\right)\right)^{\deg P \cdot\deg Q}\cdot\prod_{j=1}^{s}{\left(\frac{R_j(u.v),R_{j-1}(u,v)}{\infty}\right)},
\]
where $h(u,v)$ is a finite product of functions of the form (\ref{formesLH}) so locally constant by Lemma \ref{firststep}. In particular, one has \[\left(\frac{a}{B(u,v)}\right)=h_0(u,v)\cdot\left(\frac{a,B(u,v)}{\infty}\right),\]
where $h_0$ is a locally constant function.

Observe that, as we have seen in Lemma \ref{hilbertsymbol}, each of the $\left(\frac{R_j(u,v),R_{j-1}(u,v)}{\infty}\right)$ are locally constant on each of the connected component of $D_{R_j,R_{j-1}}=\{(u,v)\in\Z\times\Z_{<0}\ \vert \ R_j(u,v)R_{j-1}(u,v)=0\}$.

This gives:
\begin{equation}
\left(\frac{A(u,v)}{B(u,v)}\right)_\delta=\left(\prod_{i\in[0,\cdots r]\text{ ; }e_i \text{ odd}}
h_i(u,v)\right)\cdot \left(\left(\frac{u}{v}\right)_\delta\cdot
\prod_{p\mid\delta}{\left(\frac{v,u}{p}\right)}\right)^{\deg B\cdot \sum_{i=1}^{r}e_i\deg P_i }
\end{equation}
\[\cdot 
\left(\frac{a,B(u,v)}{\infty}\right)\cdot\prod_{i,j}{\left(\frac{R_{i,j}(u,v),R_{i,j-1}(u,v)}{\infty}\right)^{e_i}}
\]
where the $h_i(u,v)$ are the functions given by Lemma \ref{claim} (so finite products of functions of the form (\ref{formesLH}) ) and the $\{R_{i,j}\}_j$ are the homogeneous polynomials given by Lemma \ref{claim} for each $P_i$.

We assumed in the statement that $\deg A=\sum_{i=1}^{i=r}{e_i \cdot \deg P_i}$ is even.
For that reason, the terms $\left(\frac{u}{v}\right)_\delta \cdot \prod_{p\mid\delta}{\left(\frac{v,u}{p}\right)}$ in the decomposition of $\left(\frac{A(u,v)}{B(u,v)}\right)_\delta$ appear an even number of times, and thus simplifies to $+1$. 

By Lemma \ref{hilbertsymbol}, the last part (and thus the whole) of this product is locally constant on every connected component of \[D_{A,B}=\{(u,v)\in\Z\times\Z_{<0}\ \vert \ B(u,v)\cdot R_{1,0}(u,v)\dots R_{r,s}(u,v)=0\}.\] 
\end{proof}

However we still need to prove Lemma (\ref{claim})!

\begin{lm}\label{claim}
For any $A,B\in\Z[U,V]$ distinct primitive homogeneous polynomials, there exists $R_0,\dots R_s$ a finite number of homogeneous polynomials such that we have

\begin{equation}\left(\frac{A(u,v)}{B(u,v)}\right)_\delta=
h(u,v)\cdot \left(\left(\frac{u}{v}\right)_\delta\cdot\left(\frac{v,u}{\infty}\right)\right)^{\deg P \cdot\deg Q}\cdot 
\prod_{j=1}^{s}{\left(\frac{R_{j}(u,v),R_{j-1}(u,v)}{\infty}\right)},\end{equation}
where $h(u,v)$ is a finite product of functions of the form (\ref{formesLH}).
\end{lm}

\begin{proof}
Suppose that $A=a\in\Q$ is a constant.
Suppose that $B=b\in\Q$ is also a constant. Then by Proposition \ref{proprietes} \ref{propd}, $\left(\frac{a}{b}\right)_\delta$ is a finite product of Legendre symbols $\left(\frac{a}{p_i}\right)$.
Now suppose that $\deg B=n$. Observe that we have
\[\left(\frac{a}{B(u,v)}\right)_\delta=f_{a,B}(u,v) \left(\frac{B(u,v)}{a}\right)_\delta\]
where $f_{a,B}(u,v)$ is (by Proposition \ref{proprietes}\ref{propf}) equal to
\begin{align*}
f_{a,B}(u,v)=&\prod_{p\nmid \delta}{\left(\frac{a}{p}\right)^{\nu_p(B(u,v))}}\prod_{p\nmid \delta}{\left(\frac{B(u,v)}{p}\right)^{-\nu_p(a)}}\\
=&\prod_{p\nmid \delta; p\mid aB(u,v)}{\left(\frac{a,B(u,v)}{p}\right)^{-1}}
\prod_{p\nmid \delta; p\mid gcd(a,B(u,v))}{\left(\frac{a,B(u,v)}{p}\right)^{-1}}.
\end{align*}

The second product is finite since $a$ is fixed. The first product is also a finite product of Hilbert symbols because of the formula $\prod_{v\text{ place of }\Q}{\left(\frac{a,B(u,v)}{v}\right)}=1$. It is thus equal to 
\[\prod_{p\mid\delta, p\nmid a B(u,v)\atop \text{ or }p=\infty}{\left(\frac{a,B(u,v)}{p}\right)^{-1}}\]
Thus $\left(\frac{a}{B(u,v)}\right)_\delta$ is defined on $D_{a,B}=\{(u,v)\in\Z\times\Z_{<0} \ \vert \ B(u,v)\not=0\}$ and locally constant on each of the connected components of $\R^2-\{(u,v)\in\R^2\ \vert \ B(u,v)=0\}$.


Put $k=\deg A$, and remember that we assume that $A$ is irreducible.
The case $B=b\in\Q$ is easily deduced from Proposition \ref{proprietes}\ref{propa}: that is actually Lemma \ref{firststep}. The real first step is thus $\deg B=1$. By a linear change of variables, we can reduce to the case where $B=bU$. Thus
\begin{align}
\left(\frac{A(u,v)}{B(u,v)}\right)_{\delta}=&\left(\frac{a_0u^k+\cdots+a_kv^k}{bu}\right)_\delta&\\
=&\left(\frac{A(u,v)}{b}\right)_{\delta}\left(\frac{A(u,v)}{u}\right)_\delta &\text{by Prop \ref{proprietes} \ref{propb}}\\
=&\left(\frac{A(u,v)}{b}\right)_{\delta}\left(\frac{a_kv^k}{u}\right)_\delta &\text{by Prop \ref{proprietes} \ref{propc}}\\
=&\left(\frac{A(u,v)}{b}\right)_{\delta}\left(\frac{a_k}{u}\right)_\delta\left(\frac{v}{u}\right)_\delta^k &\text{by Prop \ref{proprietes} \ref{propb}}
\end{align}

Now, observe that \[\left(\frac{v}{u}\right)_\delta^k=\left(\frac{u}{v}\right)_\delta^k\prod_{p\mid\delta\text{ or }p=\infty}{\left(\frac{v,u}{p}\right)^{-k}}.\]

For any value of $k=\deg A$, the statement of Lemma \ref{claim} is proven. Indeed, if we put $h(u,v)=\left(\frac{A(u,v)}{b}\right)_{\delta}\cdot\left(\frac{a_k}{u}\right)_\delta\cdot\prod_{p\mid\delta}{\left(\frac{v,u}{p}\right)^{\deg A}}$, (which is a finite product of factors of the form (\ref{formesLH})) one has
\[\left(\frac{A(u,v)}{bu}\right)_\delta=h(u,v)\cdot\left(\left(\frac{u}{v}\right)_\delta\cdot\left(\frac{v,u}{\infty}\right)\right)^{\deg A}\]


Now suppose that $\deg B=l$, an integer that we suppose $l\leq k=\deg f$. Put \[B(U,V)=b_0U^l+\cdots+b_lV^l\quad\text{ and }A(U,V)=a_0U^k+\cdots+a_k V^k.\]

Then by Prop \ref{proprietes}\ref{propa}
\[\left(\frac{A(u,v)}{B(u,v)}\right)_\delta=\prod_{p\mid b_0}{\left(\frac{A(u,v)}{p}\right)^{\nu_p(B(u,v))}}\left(\frac{A(u,v)}{B(u,v)}\right)_{b_0\delta}\]
The first term is a finite product. The second term is equal to 
\[\left(\frac{b_0}{B(u,v)}\right)_{b_0\delta}\left(\frac{b_0A(u,v)}{B(u,v)}\right)_{b_0\delta}.\]
The first term is as showed earlier. The second term is equal to 
\[\left(\frac{b_0A(u,v)-a_0B(u,v)u^{k-l}}{B(u,v)}\right)_{b_0\delta}.\]
By construction the coefficient at $u^k$ of $b_0A(u,v)-a_0B(u,v)u^{k-l}$ is zero. Hence, this polynomial has the form $v^{k-k_1}A_1(u,v)$, where $k_1=\deg A_1$.

We have thus

\begin{equation}\label{h1}\left(\frac{A(u,v)}{B(u,v)}\right)_\delta=g_{1}(u,v)\cdot \left(\frac{v^{k-k_1}A_1(u,v)}{B(u,v)}\right)_{b_0\delta}\cdot \left(\frac{a_0,B(u,v)}{\infty}\right),\end{equation}
where $g_1(u,v)=\left(\frac{B(u,v)}{b_0}\right)_{b_0\delta}\cdot\prod_{p\mid\delta, p\nmid a_0 B(u,v)}{\left(\frac{a_0,B(u,v)}{p}\right)}\cdot\prod_{p\mid b_0}{\left(\frac{A(u,v)}{p}\right)^{v_p(B(u,v))}}$ is a finite product of locally constant functions. Set $R_{-1}:=A$, $k_{-1}:=\deg A$, $R_0:=B$ and $k_0:=\deg B$, and if $\deg A_1<\deg B$, set $R_1:=A_1$ and $h_1:=g_1$. If $\deg A_1\geq\deg B$, then iterate this step with $A_1$ until obtaining $A_s$, a polynomial of degree lower than $B$, and set $R_1:=A_s$ and $h_1:=\prod_{i=1}^{s}{g_i}$. Put $k_1:=\deg R_1$.

To continue the decomposition, we need to reverse the symbol with Property f of Proposition \ref{proprietes}:
\begin{align*}
\left(\frac{v^{k-k_1}R_1(u,v)}{B(u,v)}\right)_{b_0\delta}=&\prod_{p\mid\delta b_0}{\left(\frac{v^{k-k_1}R_1(u,v),B(u,v)}{p}\right)^{-1}}\left(\frac{B(u,v)}{v^{k-k_1}R_1(u,v)}\right)_{\delta b_0}\\
=&\prod_{p\mid\delta b_0 \atop\text{ or }p=\infty}{\left(\frac{v^{k-k_1}R_1(u,v),B(u,v)}{p}\right)^{-1}}\left(\frac{B(u,v)}{v}\right)_{\delta b_0}^{k-k_1}\left(\frac{B(u,v)}{R_1(u,v)}\right)_{\delta b_0}
\end{align*} 

We are now left to prove the statement for $\left(\frac{B(u,v)}{R_1(u,v)}\right)_{\delta b_0}$, a symbol with polynomials of decreased degrees. Thus, after a finite number of iterations, one reduces the problem to proving the statement for $R_r(u,v)$ of degree 1 or degree 0. That has already been proved. As the number of steps is finite, we obtain this way:
\[\left(\frac{A(u,v)}{B(u,v)}\right)_\delta=\prod_{0\leq i\leq r}{h_i(u,v)}\cdot\left(\left(\frac{u}{v}\right)_{\beta\delta}\prod_{p\mid\beta\delta\text{ or }p=\infty}{\left(\frac{v,u}{p}\right)}\right)^{\sum_{1\leq i\leq r}{k_{i-1}(k_{i-2}-k_{i})}}\cdot \prod_{1\leq i\leq r}{\left(\frac{R_i,R_{i-1}}{\infty}\right)},\]
where $h_i(u,v)$ is a product of functions of the form (\ref{formesLH}) associated to the symbol $\left(\frac{R_i}{R_{i-1}}\right)$ as in equation \ref{h1}, and $\beta\in\N^*$ is the product of the leading coefficients of the polynomials $R_i$. 

Observe that a telescopic phenomenon occurs in the exponents: \begin{align*}\sum_{1\leq i\leq r}{k_{i-1}(k_{i-2}-k_{i})}=&\deg B (\deg A-\deg R_1) + \deg R_1 (\deg B - \deg R_2)+\dots+\deg R_{s-2}(\deg R_{s-1}-R_{s})\\=&
\deg A \cdot \deg B.\end{align*}

We have that $\left(\frac{u}{v}\right)_{\delta}$ and $\left(\frac{u}{v}\right)_{\beta\delta}$ differ only by the finite product of the Legendre symbols $(\frac{u}{p})^{\nu_p(v)}$ where $p\mid\beta$. 
This achieves the proof of Lemma \ref{claim}.

\end{proof}

\nocite{silv1,silv2}

\bibliographystyle{alpha}
\bibliography{bibliothese}

\begin{thebibliography}{BCDT01}

\bibitem[BCDT01]{bcdt-fermat}
Christophe Breuil, Brian Conrad, Fred Diamond, and Richard Taylor.
\newblock On the modularity of elliptic curves over {$\mathbf{Q}$}: wild 3-adic
  exercises.
\newblock {\em J. Amer. Math. Soc.}, 14(4):843--939 (electronic), 2001.

\bibitem[CCH05]{HCC}
B.~Conrad, K.~Conrad, and H.~Helfgott.
\newblock Root numbers and ranks in positive characteristic.
\newblock {\em Adv. Math.}, 198(2):684--731, 2005.

\bibitem[CS82]{CS}
J.~W.~S. Cassels and A.~Schinzel.
\newblock Selmer's conjecture and families of elliptic curves.
\newblock {\em Bull. London Math. Soc.}, 14(4):345--348, 1982.

\bibitem[Del73]{Del}
P.~Deligne.
\newblock Les constantes des {\'e}quations fonctionnelles des fonctions {$L$}.
\newblock In {\em Modular functions of one variable, {II} ({P}roc. {I}nternat.
  {S}ummer {S}chool, {U}niv. {A}ntwerp, {A}ntwerp, 1972)}, pages 501--597.
  Lecture Notes in Math., Vol. 349. Springer, Berlin, 1973.

\bibitem[Des16a]{Desjardinsthese}
J.~Desjardins.
\newblock {\em Densit{\'e} des points rationnels sur les surfaces elliptiques
  et les surfaces de del Pezzo de degr{\'e} 1}.
\newblock PhD thesis, Universit{\'e} Paris-Diderot - Paris VII, November 2016.

\bibitem[Des16b]{Desjardins4}
J.~Desjardins.
\newblock Root number of the twists of an elliptic curve.
\newblock Submitted. Arxiv:1801.05262, 2016.

\bibitem[GM91]{GM}
F.~Gouv\^ea and B.~Mazur.
\newblock The square-free sieve and the rank of elliptic curves.
\newblock {\em Journal of the American Mathematical Society}, 4(1):1--23,
  January 1991.

\bibitem[Gre92]{Greaves}
G.~Greaves.
\newblock Power-free values of binary forms.
\newblock {\em Quart. J. Math. Oxford Ser. (2)}, 43(169):45--65, 1992.

\bibitem[GT10]{GT}
B.~Green and T.~Tao.
\newblock Linear equations in the primes.
\newblock {\em Annals of mathematics}, 171:1753--1850, 2010.

\bibitem[Hal98]{Halb}
E.~Halberstadt.
\newblock Signes locaux des courbes elliptiques en 2 et 3.
\newblock {\em C. R. Acad. Sci. Paris S{\'e}r. I Math.}, 326(9):1047--1052,
  1998.

\bibitem[Hel03]{Helfgott}
H.~A. Helfgott.
\newblock On the behaviour of root numbers in families of elliptic curves.
\newblock arXiv:math/0408141v3, 2003.

\bibitem[Hel05]{HelfChowla}
H.~A. Helfgott.
\newblock The parity problem for irreducible polynomials.
\newblock arXiv:math/0501177, 2005.

\bibitem[Hel06]{HelfChowla2}
H.~A. Helfgott.
\newblock The parity problem for reducible cubic forms.
\newblock {\em J. London Math. Soc. (2)}, 73(2):415--435, 2006.

\bibitem[Hoo67]{Hool}
C.~Hooley.
\newblock On the power free values of polynomials.
\newblock {\em Mathematika}, 14:21--26, 1967.

\bibitem[Kod63]{Kodaira}
K.~Kodaira.
\newblock On compact analytic surfaces. {I}, {II}, {III}.
\newblock {\em Ann. of Math. (2) 77 (1963), 563--626; ibid.}, 78:1--40, 1963.

\bibitem[Lac14a]{Lachandthese}
A.~Lachand.
\newblock {\em Entiers friables et formes binaires}.
\newblock PhD thesis, Universit{\'e} de Lorraine, 2014.

\bibitem[Lac14b]{Lachand1}
A.~Lachand.
\newblock Fonctions arithm{\'e}tiques et formes binaires irr{\'e}ductibles de
  degr{\'e} 3.
\newblock https://hal.archives-ouvertes.fr/hal-01053649, 2014.

\bibitem[Man95]{Manduchi}
E.~Manduchi.
\newblock Root numbers of fibers of elliptic surfaces.
\newblock {\em Compositio Math.}, 99(1):33--58, 1995.

\bibitem[N{\'e}r64]{Neronfibre}
A.~N{\'e}ron.
\newblock Mod{\`e}les minimaux des vari{\'e}t{\'e}s ab{\'e}liennes sur les
  corps locaux et globaux.
\newblock {\em Inst. Hautes {\'E}tudes Sci. Publ. Math. No.}, 21:128, 1964.

\bibitem[Neu99]{Neukirch}
J.~Neukirch.
\newblock {\em Algebraic number theory}, volume 322.
\newblock Grundlehren der Mathematischen Wissenschften, Berlin, 1999.

\bibitem[Riz03]{Rizz}
O.~G. Rizzo.
\newblock Average root numbers for a nonconstant family of elliptic curves.
\newblock {\em Compositio Math.}, 136(1):1--23, 2003.

\bibitem[Roh93]{Rohr}
D.~E. Rohrlich.
\newblock Variation of the root number in families of elliptic curves.
\newblock {\em Compositio Math.}, 87(2):119--151, 1993.

\bibitem[Ser77]{Serre}
J.-P. Serre.
\newblock {\em Cours d'arithm{\'e}tique}.
\newblock Presses Universitaires de France, Paris, 1977.
\newblock Deuxi{{\`e}}me {{\'e}}dition revue et corrig{{\'e}}e, Le
  Math{{\'e}}maticien, No. 2.

\bibitem[Sil94a]{silv2}
J.~H. Silverman.
\newblock {\em Advanced Topics in the Arithmetic of Elliptic Curves}.
\newblock Springer-Verlag, New-York, 1994.

\bibitem[Sil94b]{silv1}
J.~H. Silverman.
\newblock {\em The Arithmetic of Elliptic Curves}, volume 106.
\newblock Springer-Verlag, New-York, 1994.

\bibitem[Tat75]{Tatefibre}
J.~Tate.
\newblock Algorithm for determining the type of a singular fibre in an elliptic
  pencil.
\newblock {\em Lect. Notes in Math.}, Modular functions of one variable IV
  (Antwerpen 1972)(476):33--52, 1975.

\bibitem[Tat77]{Tat}
J.~Tate.
\newblock {\em Number theoretic background, Automorphic forms, representations
  and L-functions}.
\newblock Proc. Sympos. Pure Math., Oregon State Univ., Corvallis, Ore., 1977.

\bibitem[VA11]{VA}
A.~V{\'a}rilly-Alvarado.
\newblock Density of rational points on isotrivial rational elliptic surfaces.
\newblock {\em Algebra \& Number Theory}, 5:659--690, 2011.

\bibitem[Wil95]{Wiles}
A.~Wiles.
\newblock Modular elliptic curves and {F}ermat's last theorem.
\newblock {\em Ann. of Math. (2)}, 141(3):443--551, 1995.

\end{thebibliography}

\end{document}